\documentclass[11pt,a4paper]{amsart}
\usepackage{lmodern}
\usepackage[T1]{fontenc}
\usepackage[utf8]{inputenc}
\usepackage{amsmath}
\usepackage{amsfonts}
\usepackage{amssymb}
\usepackage{amsmath}
\usepackage{amsthm}
\usepackage{amssymb}
\usepackage{cite}	
\usepackage[shortlabels]{enumitem}	
\usepackage{xcolor}
\usepackage{tikz-cd}
\usepackage{float}
\usepackage{url}
\usepackage{hyperref}
\newtheorem{thm}{Theorem}[section]
\newtheorem{theorem}{Theorem}
\newtheorem{cor}[theorem]{Corollary}
\newtheorem{lm}[thm]{Lemma}
\newtheorem{pr}[thm]{Proposition}
\theoremstyle{definition}
\newtheorem{df}[thm]{Definition}
\theoremstyle{remark}
\newtheorem{rem}[thm]{Remark}

\newcommand{\cftf}{$C(4)$--$T(4)$}
\newcommand{\cftfs}{$C(4)$--$T(4)$ }
\newcommand{\cs}{$C(6)$}
\newcommand{\css}{$C(6)$ }
\newcommand{\ctts}{$C(3)$--$T(6)$}
\newcommand{\cttss}{$C(3)$--$T(6)$ }
\newcommand{\aot}{\cs, \cftf, and \ctts}
\newcommand{\aots}{\cs, \cftf, and \cttss}

\newcommand{\aotors}{\cs, \cftf, or \cttss}
\usepackage[a4paper, left = 2.5cm, right = 2.5cm, top = 2.5cm, bottom = 2.5cm, headsep = 1.2cm]{geometry}
\begin{document}
	
	\title[Torsion subgroups of small cancellation groups]{Torsion subgroups of small cancellation groups}

	\author[Karol Duda]{Karol Duda$^{\dag}$}
	\address{Faculty of Mathematics and Computer Science,
	University of Wroc\l aw\\
	pl.\ Grun\-wal\-dzki 2,
	50--384 Wroc\-{\l}aw, Poland}
	\email{karol.duda@math.uni.wroc.pl}
	\thanks{$\dag$ Partially supported by (Polish) Narodowe Centrum Nauki, UMO-2018/31/G/ST1/02681.}

	\begin{abstract}
	We prove that torsion subgroups of groups defined by \aotors small cancellation presentations are finite cyclic groups.
	This follows from a more general result on the existence of fixed points for locally elliptic (every element fixes a point) actions of groups on simply connected small cancellation complexes. We present an application concerning automatic continuity. We observe that simply connected \cttss complexes may be equipped with a CAT$(0)$ metric. This allows us to get stronger results on locally elliptic actions in that case. It also implies that the Tits Alternative holds for groups acting on simply connected \cttss small cancellation complexes with a bound on the order of cell stabilisers.
	\end{abstract}
	
	\maketitle

	\section{Introduction}
	The small cancellation theory is a classical powerful tool for constructing examples of groups with interesting features, as well as for exploring properties of well-known groups. It might be seen as a bridge between the combinatorial and the geometric group theories, and as one of the first appearances of nonpositive curvature techniques in studying groups. Constituting a classical part of mathematics, small cancellation techniques are still being developed, having numbers of variations, and have also been used in proving new important results nowadays. We direct the reader to the classical book of Lyndon-Schupp~\cite{ls} for basics on small cancellation.
	
	Among many notions of small cancellation, the best known and the most deeply explored ones are the so-called ``combinatorial small cancellation conditions": \aot. Seeing small cancellation as an example of nonpositive curvature, and motivated by a number of well-known conjectures and open questions concerning torsion subgroups of ``nonpositively curved" groups we prove the following.
	
	\begin{theorem}
		\label{thm:tA}
		Torsion subgroups of groups defined by \aots small cancellation presentations are finite cyclic groups.
	\end{theorem}

	The theorem may be seen as a small cancellation counterpart of a (still) conjectural feature of CAT$(0)$ groups (see \cite{Swe99}). Note however that in our result we allow the presentations to be infinite. 
	We direct the reader to an extensive discussion in \cite{HaeOsa} on torsion subgroups and related \emph{locally elliptic} (that is, every element fixes a point) actions on ``nonpositively curved" complexes. Theorem~\ref{thm:tA} establishes a particular case of \cite[Meta-Conjecture and Conjecture]{HaeOsa}.

	In fact, it is conjectured in \cite{HaeOsa} that locally elliptic actions of finitely generated groups on finitely dimensional nonpositively curved complexes (in particular, on simply connected small cancellation complexes) have global fixed points. 
	This problem is strongly related to various other questions concerning nonpositively curved groups, in particular to the automatic continuity (see Theorem~\ref{thm:tC} below) and to the Tits Alternative (see e.g.\ \cite{OsPrz21,OsPrz22} for recent advancements concerning the Tits Alternative for small cancellation groups and for further references).
	Our approach goes along the same path: Theorem~\ref{thm:tA} is an immediate consequence of the following result.

	\begin{theorem}
		\label{thm:tB}
		Let $X$ be a simply connected \aotors small cancellation complex. Let $G$ be a group acting on $X$ by automorphisms such that the action induces a free action on the $1$-skeleton $X^{1}$ of $X$.
		If the action is locally elliptic, then $G$ is a finite cyclic group.
		In particular, $G$ fixes a $2$-cell of $X$.
	\end{theorem} 
	
	Although we believe that some versions of Theorem~\ref{thm:tB} hold without the assumptions on the freeness of the action, in the current statement these cannot be omitted. For example, it has been shown by Serre \cite[Theorem~15, Chapter I.6.1]{Serre_Trees} that any countable infinitely generated group acts without fixed points on a tree, which is a $1$-dimensional small cancellation complex (the conjectures in \cite{HaeOsa} concern mostly finitely generated groups). 
	Furthermore, our assumptions are tailored for the application to Theorem~\ref{thm:tA}.
		
	In general, there is no known way of equipping a small cancellation complex with a ``reasonable" CAT$(0)$ structure (and some experts doubt it can be done at all). Nevertheless, we observe that \cttss complexes admit such a structure.
	This is a relatively simple observation following a remark by Pride \cite[p.165]{Pride}. Although the next result implies a number of significant features of \cttss groups, it seems it has not been observed before.
	
	\begin{theorem}
	\label{thm:tF}
	Let $X$ be a simply connected \cttss  small cancellation complex. Then there exists a metric on $X$ turning it into
	a CAT$(0)$ triangle complex $\mathfrak{X}$ such that every automorphism of $X$ induces an automorphism of $\mathfrak{X}$.
	\end{theorem}

	The existence of a CAT$(0)$ structure allows us to extend Theorem~\ref{thm:tB} in the \cttss case, using results of \cite{NOP-D}.

	\begin{cor}\label{thm:tE}
	A finitely generated group acting locally elliptically on a simply connected \cttss  small cancellation complex fixes a point.
	\end{cor}

	From the CAT$(0)$ property of \cttss complexes we also conclude the following result closely related to the non-existence of infinite torsion subgroups. Recall that a group satisfies the \emph{Tits Alternative} if each of its finitely generated subgroups either contains a free nonabelian subgroup or is virtually solvable. It is believed that ``nonpositively curved" groups satisfy the Tits Alternative, but this has been proved only in a limited number of cases. See \cite{SW2005,OsPrz21,OsPrz22} for more details. The following theorem states that the Tits Alternative holds for groups acting \emph{almost freely} (there is a bound on the order of cell stabilisers) on \cttss complexes.		
	
	\begin{cor}
	\label{thm:tD}
	Let $G$ be a group acting almost freely on a simply connected \cttss small cancellation complex.
	Then $G$ is virtually cyclic, or virtually $\mathbb{Z}^2$, or contains a nonabelian free group.
\end{cor} 	

	Finally, let us present an application of Theorem~\ref{thm:tA} to \emph{automatic continuity}. This property has its origins in the descriptive set theory and, roughly, says that every group homomorphism between topological groups is continuous. Recently automatic continuity has been established for a large class of homomorphisms into ``nonpositively curved" groups, see e.g.\ \cite{keppeler2021automatic} -- we extend some results of that paper in the following, where $G$ is equipped with a discrete topology. 
	
	\begin{theorem}
		\label{thm:tC}
		Let $G$ be a group acting geometrically on a locally finite, simply connected \aotors small cancellation complex $X$ such that the action induces a free action on the $1$-skeleton of $X$. If $H$ is a subgroup of $G$ then any group homomorphism $\varphi : L \rightarrow H$ from a locally compact group $L$ is continuous or there exists a normal open subgroup $N\subseteq L$ such that $\varphi(N)$ is a finite group.
	\end{theorem}	

\subsection*{Related results} 
	For finitely generated groups acting on uniformly locally finite simply connected \cftfs complexes a version of Theorem~\ref{thm:tB} has been proved recently in \cite[Corollary B(3)]{HaeOsa}.
The proof there uses the fact (established in \cite{Helly}) that groups acting geometrically on simply connected \cftfs small cancellation complexes are Helly.
In particular, in the case of finitely presented groups the \cftfs part of Theorem~\ref{thm:tA} follows from \cite[Corollary B(3)]{HaeOsa}. We do not
need such finiteness assumptions for our proof. As for further results establishing the Conjecture and, more generally, the (Meta-)Conjecture from \cite{HaeOsa} in specific cases of small cancellation let us mention the case of CAT$(0)$ square complexes (an example of \cftf) following from \cite{Sageev1995,leder_varghese}, the case of $2$-dimensional systolic complexes (an example of $C(3)$--$T(6)$) and other $2$-dimensional CAT$(0)$ complexes from \cite{NOP-D}, the case of (graphical) $C(18)$ complexes from \cite[Theorem F]{HaeOsa}, and the case of $C'(1/4)$--$T(4)$ complexes from \cite{genevois2016coningoff}.

Whether the Tits Alternative holds for small cancellation groups is an open problem. In \cite{Col73,ALj77,EH88}  the Weak Tits Alternative is shown, see \cite{SW2005} for a discussion. 

The \cftfs case of Theorem~\ref{thm:tC} has been established in \cite[Corollary H]{HaeOsa}.

\subsection*{Idea of the proof of the Theorem \ref{thm:tB}} 
We restrict here to \cs, and \cftfs cases. First, we show that each element of the group $G$ fixes the center of exactly one $2$-cell. Therefore $G$ does not have elements of infinite order, as such elements would not act freely on the $1$-skeleton of $X$. 
Then we show that if $G$ acts on $X$ without a global fixed point, then $G$ has an element of infinite order, contradiction. 

To prove that latter implication, consider two elements $f,g\in G$ with $\mathrm{Fix}(f)\neq \mathrm{Fix}(g)$. 
We consider a ``dual" complex $Y$ of $X$: it is the quadrization in the \cftfs case, and the Wise complex in the \css case. These complexes are: quadric and systolic, respectively.
In particular, $Y$ is simply connected and $G$ acts by automorphisms on $Y$. 

Both $\mathrm{Fix}_Y(f)$ and $\mathrm{Fix}_Y(g)$ consist of one vertex each, we denote them by $x$ and $y$. 
Since $Y$ is connected, we can find a geodesic $\gamma := (x_0 :=y, x_1,\ldots, x_n:=x)$ in $Y$. 
We find $k,l$ such that $x_1$ (resp. $x_{n-1}$) is at distance $2$ (\cftfs case) or $3$ (\css case) from $g^lx_1$ (resp. $f^kx_{n-1}$) in the link of $y$ (resp. $x$).
For such $k,l$ and any $i$, we show that the path $\alpha_i := \bigcup\limits_{0\leq j\leq i}(f^kg^l)^j(\gamma\cup f^k\gamma)$ is a geodesic. Therefore $f^kg^l$ has infinite order.

\subsection*{Structure of the paper}
In Section \ref{sc} we give a brief introduction to small cancellation theory.
Sections \ref{qd}-\ref{inford} concern the case of \cftfs and \css complexes.
In Section \ref{qd} we define the quadrization of a complex and show properties of quadrizations of \cftfs complexes. 
In Section \ref{sys} we define and show properties of the Wise complexes of \css complexes. 
In Section \ref{cv} we discuss properties of curvature of a CAT(0) square and simplicial disc diagrams.
In Sections \ref{negcuv} and \ref{rotat} we prove technical lemmas about \cftfs and \css complexes, necessary for Section \ref{inford}. 
In Section \ref{inford} we show that if a group $G$ acts on a simply connected \cftfs or \css complex $X$ as in Theorem \ref{thm:tB} then the lack of a global fixed point implies that there is an element of infinite order in $G$.
Section \ref{catct} concerns the case of \cttss complexes, in particular proofs of Theorem \ref{thm:tF} and Corollaries \ref{thm:tE}-\ref{thm:tD} are in this section.
In Section \ref{sec: GASCC} we finish the proof of Theorem \ref{thm:tB} and prove Theorem \ref{thm:tC}.
In Section \ref{sec: PTA} we prove Theorem \ref{thm:tA}.
\subsection*{Acknowledgements} The author would like to thank Daniel Danielski, Nima Hoda, Damian Osajda and Motiejus Valiunas whose comments led to many corrections and improvements.

\section{Basic definitions}
The purpose of this section is to give basic definitions and terminology regarding combinatorial $2$-complexes. We follow the notation of \cite{MW}.
For fundamental notions such as \textit{CW}-\textit{complexes}, \textit{nullhomotopy} and \textit{simple connectedness}, see Hatcher's textbook on algebraic topology \cite{AH}. In this paper, we consider only $2$-dimensional CW-complexes and we will refer to them as ``$2$-complexes''. Throughout this section, we assume that $X$ and $Y$ are $2$-dimensional CW-complexes.

A map from $X$ to $Y$ is \textit{combinatorial} if it is a continuous map whose restriction to every open cell $e$ of $X$ is a homeomorphism from $e$ to an open cell of $Y$. A complex is \textit{combinatorial} if the attaching map of each of its $n$-cells is combinatorial for a suitable subdivision of the sphere $\mathbb{S}^{n-1}$. An \textit{immersion} is a combinatorial map that is locally injective.

Unless stated otherwise, all maps and complexes are combinatorial, and all attaching maps are immersions.

A \textit{polygon} is a $2$-disc with a cell structure that consists of $n$ $0$-cells, $n$ $1$-cells and a single $2$-cell. For any $2$-cell $C$ of a $2$-complex $X$ there exists a map $R\rightarrow X$ where $R$ is a polygon and the attaching map for $C$ factors as $\mathbb{S}^1\rightarrow \partial R \rightarrow X$. Because of that, by a \textit{cell}, we will mean a map $R\rightarrow X$ where $R$ is a polygon. 

Another important notion in this paper are \textit{simplicial} complexes. These square complexes are $n$-dimensional complexes whose $n$-cells are $n$-simplices. As already mentioned, in this paper we only consider $2$-complexes. In the case of simplicial complexes, instead of the usual terms, $0$-cell, $1$-cell and $2$-cell, we will use vertex, edge and triangle, respectively.

\sloppy Let $X$ be a simplicial complex and $L$ be a subcomplex of $X$ with vertices $\{u_1,\ldots, u_{n},v_1,\ldots, v_{n+1}\}$,
and the set of edges consisting of the edges of the form $(u_i,u_{i+1}),(v_i,v_{i+1})$, $(u_i,v_i)$ and $(u_i,v_{i+1})$.
We call such a complex $L$ a \textit{ladder of length} $n$ and denote it by $\{u_1,\ldots, u_n|v_1,\ldots, v_{n+1}\}$.

\begin{figure}[H]
\begin{center}
\includegraphics[scale=1.2]{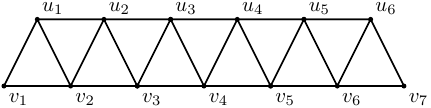}
\end{center}
\caption{Ladder of length $6$.}\label{fig:ladt}
\end{figure}

Another important for us type of complexes are \textit{square} complexes. Square complexes are $2$-complexes whose $2$-cells are $4$-gons. In this case, instead of the usual terms, $0$-cell, $1$-cell and $2$-cell, we will use vertex, edge and square, respectively.

Let $X$ be a square complex and $L$ be a subcomplex of $X$ with vertices $\{u_1,\ldots, u_n,v_1,\ldots, v_n\}$,
and the set of edges consisting of the edges of the form $(u_i,u_{i+1}),(v_i,v_{i+1})$ and $(v_i,u_i)$.
We call such a complex $L$ a \textit{ladder of length} $n$ and denote it by $\{u_1,\ldots, u_n|v_1,\ldots, v_n\}$.

\begin{figure}[H]
\begin{center}
\includegraphics[scale=1.2]{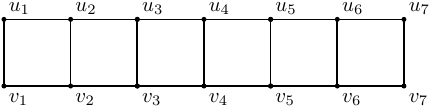}
\end{center}
\caption{Ladder of length $7$.}\label{fig:lads}
\end{figure}

A \textit{path} in $X$ is a combinatorial map $P\rightarrow X$ where $P$ is either a subdivision of the interval or a single $0$-cell. For given paths $P_1\rightarrow X$ and $P_2\rightarrow X$ such that the terminal point of $P_1$ is equal to the initial point of $P_2$, their \textit{concatenation} is the natural path $P_1P_2\rightarrow X$. 
Similarly, a \textit{cycle} is a map $C\rightarrow X$ where $C$ is a subdivision of a circle $\mathbb{S}^1$. 
It is clear that a closed path, i.e. a path such that the initial point is equal to the terminal point, is a cycle.
We will often identify paths and cycles with their images in the complex $X$.

A \textit{disc diagram} is a contractible finite $2$-complex $D$ with a specified embedding into the plane. The \textit{area} of diagram $D$ is the number of its $2$-cells. If $D$ is a disc diagram, then the diagram $D$ \textit{in} $X$ is $D$ along with a combinatorial map from $D$ to $X$ denoted by $D\rightarrow X$. 

We say that a disc diagram $D$ is \textit{nonsingular} if it is homeomorphic to a closed $2$-cell and \textit{singular} otherwise. If a disc diagram $D$ is singular, then it is either trivial, or it consists of a single $1$-cell joining two $0$-cells, or it contains a \textit{cut} $0$-\textit{cell (or cut-vertex)}, i.e.~a $0$-cell $v$ of $D$ such that $D\setminus\{v\}$ is disconnected.

A \textit{boundary cycle} $\partial D$ of $D$ is the attaching map of the $2$-cell that contains the point $\infty$ when we regard $\mathbb{S}^2=\mathbb{R}^2\cup \{\infty\}$.

We also have to state some notions concerning graphs.

A \textit{graph} is a $1$-dimensional CW-complex with $0$-cells called \textit{vertices} and $1$-cells called \textit{edges}. Distances between vertices in graphs are always measured by the \textit{standard graph metric} which is defined for a pair of vertices $u$ and $v$ as the number of edges in the shortest path connecting $u$ and $v$. 

Graph $\Gamma$ is \textit{simplicial} if there is no edge in $\Gamma$ with both endpoints attached to one vertex and no two edges of $\Gamma$ having their endpoints attached to the same unordered pair of vertices.
Let $\Gamma$ be a graph with the set of vertices $V$ and the set of edges $E$. Let $V'$ be a subset of the set of vertices $V$. The subgraph $\Gamma'$ of $\Gamma$ consisting of the set of vertices $V'$ and the set of edges $E':=\{(u,v)|u,v\in V', (u,v)\in E\}$ is called the \textit{subgraph induced by} $V'$.

Among class of simplicial graphs we distinguish so called \textit{bipartite} graphs: a graph $\Gamma$ is bipartite if the set of its vertices can be partitioned into two nonempty sets such that no edge has both of its endpoints in the same set. 

A \textit{link} of a $0$-cell $v$ of a $2$-complex $X$ (denoted $X_v$) is the graph whose vertices correspond to the ends of $1$-cells of $X$ incident to $v$, and an edge joins vertices corresponding to the ends of $1$-cells $e_1, e_2$ iff there is a $2$-cell $F$ such that $e_1,e_2 \in \partial F$.

\section{Small cancellation}\label{sc}

The following definition is crucial in the small cancellation theory.
\begin{df}
Let $X$ be a combinatorial $2$-complex. A non-trivial path $P\rightarrow X$ is a \textit{piece} of $X$ if there are $2$-cells $R_1$ and $R_2$ such that $P\rightarrow X$ factors as $P\rightarrow R_1\rightarrow X$ and as $P\rightarrow R_2\rightarrow X$ but there does not exist a homeomorphism $\partial R_1\rightarrow \partial R_2$ such that
there is a commutative diagram:

$$
  \begin{tikzcd}
    P \arrow{r} \arrow{d} & \partial R_2 \arrow{d} \\
    \partial R_1 \arrow{r} \arrow[swap]{ur} & X
  \end{tikzcd}
$$
\end{df}

Intuitively, a piece of $X$ is a path which is contained in boundaries of $2$-cells of $X$ in at
least two distinct ways.

\begin{df}
Let $X$ be a $2$-complex. A disc diagram $D\rightarrow X$ is \textit{reduced} if for every piece $P\rightarrow D$ the composition $P \rightarrow D\rightarrow X$ is a piece in $X$.  
\end{df}

The following theorem is known as the Lyndon-van Kampen lemma \cite[Lemma 2.17]{MW}.

\begin{thm}\label{lvk}
If $X$ is a $2$-complex and $P\rightarrow X$ is a nullhomotopic closed path, then there exists a reduced disc diagram $D\rightarrow X$ such that $P\rightarrow D$ is the boundary cycle of $D$, and $P\rightarrow X$ is the composition $P\rightarrow D\rightarrow X$.
\end{thm}

Let $\alpha$ be a nullhomotopic cycle in $X$. We say that disc diagram $D$ is a \textit{minimal area disc diagram for $\alpha$ in $X$} if $\alpha$ is a boundary cycle of $D$ and $D$ has the smallest area amongst all disc diagrams with boundary cycle $\alpha$. 

We will now define small cancellation conditions $C(p)$ and $T(q)$.

\begin{df}
Let $p$ be a natural number. We say that a $2$-complex $X$ satisfies the $C(p)$ \textit{small cancellation condition} provided that for each $2$-cell $R \rightarrow X$ its attaching map $\partial R\rightarrow X$ is not a concatenation of fewer than $p$ pieces in $X$.
\end{df}
In the following definition of $T(q)$ condition we will use the notion of \textit{valence} of a $0$-cell $v\in X$ in the complex $X$, i.e.~the number of ends of $1$-cells incident to it. We denote it by $\delta_{X}(v)$ and drop the subscript if it is clear from the context.
\begin{df}
Let $q$ be a natural number. We say that a $2$-complex $X$ satisfies the $T(q)$ \textit{small cancellation condition}  if there does not exist a reduced map $D\rightarrow X$ where $D$ is a disc diagram containing an internal $0$-cell $v$ such that $2<\delta(v)<q$.  
\end{df}

If a complex satisfies both $C(p)$ and $T(q)$ conditions, then we call it a $C(p)$--$T(q)$ complex. As mentioned already, in this paper we are working with \aotors complexes. We will now state some known properties of these complexes.

The following proposition is a known property of simply connected \css complexes \cite{OP18}.

\begin{lm}\label{ver}
Let $\widehat{x}_1,\widehat{x}_2,\ldots, \widehat{x}_n$ be a pairwise intersecting $2$-cells of a \css complex $X$. Then $\widehat{x}_1\cap \widehat{x}_2\cap\ldots\cap \widehat{x}_n$ is either a piece or a vertex.
\end{lm}

The following propositions are known properties of simply connected \cftfs complexes \cite[Proposition 3.5, 3.8]{hoda2019quadric}.

\begin{pr}\label{ctc}
Let $F_1$ and $F_2$ be a pair of intersecting $2$-cells of a simply connected \cftfs complex. Then the intersection $F_1\cap F_2$ is connected. 
\end{pr}

\begin{pr}[Strong Helly Property] 
Let $F_1$, $F_2$ and $F_3$ be pairwise intersecting $2$-cells of a simply connected \cftfs complex. Then the intersection of some pair of these $2$-cells is contained in the remaining $2$-cell, i.e.~for some permutation $\sigma$ of the indices, we have

$$F_{\sigma(1)}\cap F_{\sigma(2)}\subset F_{\sigma(3)}.$$
\end{pr}

In the case of \cttss complexes, we will use the fact that all pieces are short, which was first observed by Pride \cite{Pride} in the following Lemma.

\begin{lm}
If $X$ is a $T(q)$ complex for $q\geq 5$, then every piece in $X$ has length $1$.
\end{lm}

\section{Quadric complexes and quadrization of a complex}\label{qd}

In this section we define the quadrization of a complex and state some results concerning quadrizations of \cftfs complexes. We begin with some necessary notions concerning CAT(0) square complexes.

A \textit{CAT(0) square complex} is a square complex for which the metric obtained by making each square isometric to the regular Euclidean square of side length $1$ satisfies the CAT(0) condition, which is a metric nonpositive curvature condition concerning thinness of geodesic triangles. We use a combinatorial characterization of the CAT(0) condition for the square complexes which follows from Gromov's link condition, and take it as the definition.
\begin{df}
A square complex $X$ is CAT(0) if it is simply connected and for any $0$-cell $v\in X$ the shortest embedded cycle in the link of $v$ has length at least $4$.
\end{df}
If $X$ is a disc diagram, the CAT(0) condition means that each internal vertex is incident to at least four squares.

Let $X$ be a $2$-complex with embedded $2$-cells. Let $X_0, X_2$ be the sets of $0$-cells and $2$-cells of $X$. Let $\Gamma_X$ be a bipartite graph on the vertex set $X_0\cup X_2$ where an edge joins $v\in X_0$ with $F\in X_2$ whenever $v\in \partial F$. The following notion was introduced by N.Hoda.
\begin{df}\cite[Section 3.2]{hoda2019quadric}
The \textit{quadrization} $Y$ of a complex $X$ is the $4$-flag completion $Y=\overline{\Gamma_X}$ i.e.~a complex obtained from $\Gamma_X$ by spanning a $2$-cell on each of its nontrivial $4$-cycles. 
\end{df}

If we additionally assume that every $1$-cell of $X$ is contained in the boundary of a $2$-cell then simply connectedness of $X$ implies simply connectedness of its quadrization \cite[Lemma 3.9]{hoda2019quadric}.

All $2$-cells are squares, therefore the quadrization of a complex is a square complex. In this paper sometimes we will consider cells in both the complex $X$ and its quadrization $Y$ at the same time. In such cases we will denote by $\widehat{x}$ the $0$- or $2$-cell corresponding to the vertex $x$ from $Y$. 

A path $P\rightarrow Y$ will be denoted by $(x_1,\ldots,x_n)$, where $x_i$ are the vertices in $Y$ that are the images of vertices of $P$, in particular, $x_1,x_n$ are the images of ends. If a path is a piece, then we will denote it by $(x,x')$ where $x$ and $x'$ are its endpoints. A cycle $C\rightarrow Y$ will be denoted by $\langle x_1,\ldots,x_n\rangle$.

Each square $F$ in $Y$ has four vertices: two from $X_0$ and two from $X_2$. Boundary of $F$ is a cycle $\partial F=\langle x_1,x_2,x_3,x_4 \rangle$. We will use notation $F=[x_1,x_2,x_3,x_4]$. Observe that either $x_1,x_3\in X_0$ and $x_2,x_4\in X_2$, or $x_1,x_3\in X_2$ and $x_2,x_4\in X_0$; wherever possible we will use the first of these options. Let $F=[x_1,x_2,x_3,x_4]$ be a square in $Y$. Since $X$ is \cftf, the intersection of $\widehat{x}_2,\widehat{x}_4$ is connected and contains $\widehat{x}_1,\widehat{x}_3$. Therefore $(\widehat{x}_1,\widehat{x}_3)$ is a piece in $X$. On the other hand, if $(\widehat{x}'_1,\widehat{x}'_3)$ is a piece in $X$ contained in $\widehat{x}_2'\cap \widehat{x}_4'$, then  $F'=[x'_1,x'_2,x'_3,x'_4]$ is a square in $Y$.

In \cite{hoda2019quadric} N.Hoda proved that the $1$-skeleton of the quadrization of a simply connected \cftfs complex is a hereditary modular graph. In his paper Hoda called such complexes \textit{quadric}. We only need a part of his result. Hoda defines quadric complexes as simply connected square complexes satisfying, among others, the following conditions, called \textit{rules of replacement}.  

\begin{pr}\cite[Definition 1.1.]{hoda2019quadric}
Let $Y$ be the quadrization of the \cftfs complex $X$.
\begin{enumerate}[(A)]
\item If there are two squares $F_1,F_2\in Y$ such that $\partial(F_1\cup F_2)$ is a cycle of length $4$, then there is $F\in Y$ such that $\partial F=\partial(F_1\cup F_2)$.
\item If there are three squares $F_1,F_2,F_3\in Y$ such that $\partial(F_1\cup F_2\cup F_3)$ is a cycle of length $6$, then there exist $F,F'\in Y$ such that $\partial(F_1\cup F_2\cup F_3)=\partial(F\cup F')$, i.e.~this cycle has a diagonal that divides it into two $4$-cycles.
\end{enumerate}
\end{pr}

\begin{figure}[h]
\begin{center}
\includegraphics[scale=2.2]{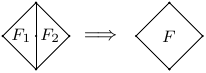}
\end{center}
\begin{center}
(A)
\end{center}
\begin{center}
\includegraphics[scale=2.2]{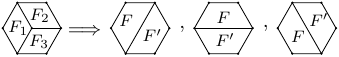}
\end{center}
\begin{center}
(B)
\end{center}
\caption{Replacement rules for quadric complexes}
\end{figure}

\begin{pr}\label{quad}
Let $Y$ be the quadrization of a \cftfs complex $X$ and $\alpha$ be some cycle in $Y$. If $D$ is a minimal area disc diagram for $\alpha$ in $Y$, then $D$ is a CAT(0) square complex.
\end{pr}
\begin{proof}
Suppose that there is an internal vertex $v$ in $D$ incident to $k<4$ squares. Since $v$ is internal, $k$ is either $2$ or $3$. Let $D'$ be the union of squares incident to $v$.

Suppose $k=2$. Let $F_1=[v,x_1,x_2,x_3], F_2=[v,x'_1,x'_2,x'_3]$. Since $v$ is an internal vertex, without loss of generality $x_1=x'_1, x_3=x'_3$. The vertex $x_2$ cannot be the same vertex as $x'_2$ as then both $F_1$ and $F_2$ would be spanned by the same $4$-cycle what contradicts the definition of quadrization. Therefore $\langle x_2,x_1,x'_2,x_3\rangle= \partial D'$. By the first rule of replacement, it spans a square that contradicts the minimality of the area of $D$.

Suppose that $k=3$. Without loss of generality we may assume that every internal vertex of $D$ is incident to at least three squares. Then $D'$ consists of $3$ squares bounded by a $6$-cycle. By the second rule of replacement there exists diagram $D''$ such that $\partial D''=\partial D'$ and $D''$ consists of two squares, a contradiction to minimality of the area of $D$.
\end{proof}

Let $Y$ be the quadrization of a \cftfs complex.
Let $L$ be the subcomplex of $Y$ with the set of vertices $$\{u_1,\ldots, u_n,v_1,\ldots, v_n, w_1,\ldots, w_n, c\},$$ the set of edges consisting of the edges of the form $(u_i,u_{i+1}),(v_i,v_{i+1}),(w_i,w_{i+1}),(w_i,v_i), (v_i,u_i)$ and $(c,w_n), (c,u_n)$ and the set of squares consisting of squares $[u_i,u_{i+1},v_i,v_{i+1}]$, $[v_i,v_{i+1},w_i,w_{i+1}]$ and $[c,u_n,v_n,w_n]$.
We call such a complex a \textit{double ladder of length} $n$ \textit{with a cap} and denote it by $\{u_1,\ldots, u_n|v_1,\ldots, v_n| w_1,\ldots, w_n| c\}.$ The square $[c,u_n,v_n,w_n]$ is called a \textit{cap} (see Fig. \ref{fig:dlc}).

\begin{figure}[H]
\begin{center}
\includegraphics[scale=1]{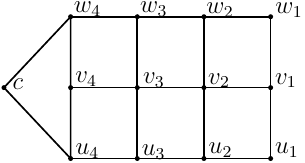}
\end{center}
\caption{Double ladder $\{u_1,\ldots, u_4|v_1,\ldots, v_4| w_1,\ldots, w_4| c\}$ of length $4$ with a cap.}\label{fig:dlc}
\end{figure}

\begin{pr}\label{dlwc}
If $\{u_1,\ldots, u_n|v_1,\ldots, v_n| w_1,\ldots, w_n| c\}$ is a double ladder with a cap, then at least one of the following conditions holds:
\begin{enumerate}[1)]
\item  $(c, v_{n-1})\in Y$;
\item for some $i \in \{3,\ldots, n\}$ $(u_i, v_{i-2})\in Y$;
\item for some $i \in \{3,\ldots, n\}$ $(w_i, v_{i-2})\in Y$;
\item $(u_2, w_{1})\in Y$; 
\item $(u_1, w_{2})\in Y$.
\end{enumerate}
\end{pr}

\begin{proof}
Assume that none of these holds. By definition of a double ladder with a cap, $\langle c, w_n,w_{n-1},v_{n-1},u_{n-1},u_n\rangle$ is a $6$-cycle in $Y$, that bounds three squares. By the second rule of replacement there exists a diagonal splitting into two $4$-cycles. Therefore one of $(c, v_{n-1}), (u_n, w_{n-1}), (u_{n-1}, w_{n})$  belongs to $Y$; see Fig. \ref{fig:dlcp1} 

\begin{figure}[H]
\begin{center}
\includegraphics[scale=1]{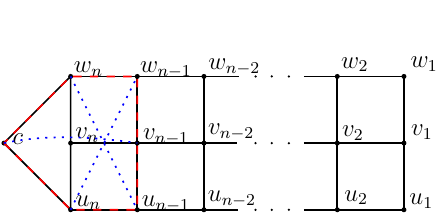}
\end{center}
\caption{The $6$-cycle $\langle c, w_n,w_{n-1},v_{n-1},u_{n-1},u_n\rangle$ bounds three squares in red. One of blue edges belongs to $Y$.}\label{fig:dlcp1}
\end{figure}

Since 1) does not hold, we obtain a shorter double ladder with the cap being $[w_n,u_{n-1},v_{n-1},w_{n-1}]$ or $[u_n,u_{n-1},v_{n-1},w_{n-1}]$.
\begin{figure}[H]
\begin{center}
\includegraphics[scale=1]{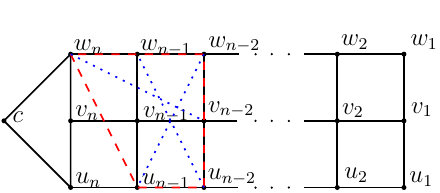}
\end{center}
\caption{One of the possible options for the second $6$-cycle that bounds three squares in red. One of the blue edges belongs to $Y$.}\label{fig:dlcp2}
\end{figure}

Since there is no $i$ such that 2) or 3) holds, by induction we obtain that $(u_2, w_{1})\in Y$ or $(u_1, w_{2})\in Y$, a contradiction.
\end{proof}

\section{Systolic complexes and Wise complex}\label{sys}

In this section we define systolic complexes and state some results concerning the Wise complex of a \css complex.

A \textit{CAT(0) simplicial complex} is a simplicial complex for which the metric obtained by making each triangle isometric to the equilateral Euclidean triangle satisfies the CAT(0) condition. As in the square complex case we use a combinatorial characterization, following from Gromov's link condition, as the definition.

\begin{df}
A simplicial complex $X$ is CAT(0) if it is simply connected and for any $0$-cell $v\in X$ the shortest embedded cycle in the link of $v$ has length at least $6$.
\end{df}

If $X$ is a disc diagram, the CAT(0) condition means that each internal vertex is incident to at least six triangles.

Let $X$ be a $2$-complex with embedded $2$-cells, such that every $1$-cell is contained in the boundary of a $2$-cell.

Let $\mathcal{U}$ be a cover of $X$.
The \textit{nerve} of the cover $\mathcal{U}$ is a simplicial complex whose vertex set is $\mathcal{U}$, and vertices $U_0,\ldots, U_n$ span an $n$-simplex if and only if $\bigcap\limits_{0\leq i\leq n}U_i\neq \emptyset$.

\begin{df}
The \textit{Wise complex} $Y$ of a complex $X$ is the nerve of the covering of $X$ by closed $2$-cells.
\end{df}

By definition Wise complex is a simplicial complex. 
Similarly as in the case of quadrization, we will sometimes consider some cells in both the complex $X$ and its Wise complex $Y$ at the same time. In such cases by $\widehat{x}$ we will denote the $2$-cell corresponding to the vertex $x$ from $Y$. 

It is known that the Wise complex of a simply connected \css complex is systolic \cite[Theorem 10.6]{wise}, see also \cite{OP18}.

\begin{df}
Let $X$ be a simplicial complex. 
$X$ is $k$-\textit{large} if $X$ is flag and every cycle in $X$ of length less than $k$ has a diagonal, i.e. in $X$ there is an edge connecting nonconsecutive vertices of the cycle.
\end{df}

\begin{lm}\cite[Lemma 1.3]{TJJS}\label{13tjjs}
Suppose that $X$ is $k$-large and $\mathbb{S}^1_m$ denotes the triangulation of $\mathbb{S}^1$ with $m$ $1$-cells. If $m<k$ then any simplicial map $f:\mathbb{S}^1_m\rightarrow X$ extends to a simplicial map from disc $\mathbb{D}^2$, triangulated so that the triangulation on the boundary is $\mathbb{S}^1_m$ and so that there are no interior vertices in $\mathbb{D}^2$.
\end{lm}

\begin{df}
A simplicial complex $X$ is \textit{systolic} if $X$ is simply connected, and links of all vertices of $X$ are $6$-large.
\end{df}

Proposition below follows from minimal area disc diagrams of systolic complexes being systolic and a remark by Januszkiewicz and \'Swi\k{a}tkowski \cite[p.51]{TJJS}.

\begin{pr}
Let $\alpha$ be a cycle in a systolic complex $X$. If $D$ is a minimal area disc diagram for $\alpha$ in $X$, then $D$ is a CAT(0) simplicial complex.
\end{pr}

Let $L$ be the flag subcomplex of $Y$ spanned by the set of vertices 
$$\{u_1,\ldots, u_n,v_1,\ldots, v_n, w_1,\ldots, w_n\}$$
and the set of edges: 
$$
\{(u_i,u_{i+1}),(v_i,v_{i+1}),(w_i,w_{i+1}),(w_i,v_{i+1}),(v_{i+1},u_i): 1\leq i\leq n-1\}\cup$$
$$ \{(w_i,v_i), (v_i,u_i): 1\leq i\leq n\}\cup\{(u_n,w_n)\}.
$$  
We call such a complex a \textit{double ladder of length} $n$ \textit{with a cap} and denote it by $$\{u_1,\ldots, u_n|v_1,\ldots, v_n| w_1,\ldots, w_n\}.$$ The triangle $[u_n,v_n,w_n]$ is called a \textit{cap}, see Fig. \ref{fig:dlcsys}.

\begin{figure}[H]
\begin{center}
\includegraphics[scale=1]{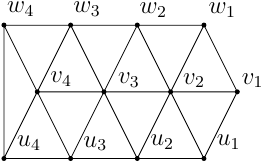}
\end{center}
\caption{Double ladder $\{u_1,\ldots, u_4|v_1,\ldots, v_4| w_1,\ldots, w_4\}$ of length $4$ with a cap.}\label{fig:dlcsys}
\end{figure}

\begin{pr}\label{dlwcsys}
If $\{u_1,\ldots, u_n|v_1,\ldots, v_n| w_1,\ldots, w_n\}$ is a double ladder with a cap, then at least one of the following conditions holds:
\begin{enumerate}[1)]
\item for some $i \in \{3,\ldots n\}$ $(u_i, v_{i-1})\in Y$;
\item for some $i \in \{3,\ldots n\}$ $(w_i, v_{i-1})\in Y$;\item $(u_1, w_{1})\in Y$.
\end{enumerate}
\end{pr}

\begin{proof}
Assume that none of these holds. By definition of a double ladder with a cap, $\langle w_n,w_{n-1},v_{n-1},u_{n-1},u_n\rangle$ is a $5$-cycle in $Y$. By systolicity, one of $(w_n, v_{n-1}), (u_n, v_{n-1}), (u_{n-1}, w_{n-1})$ belongs to $Y$. 

\begin{figure}[H]
\begin{center}
\includegraphics[scale=1]{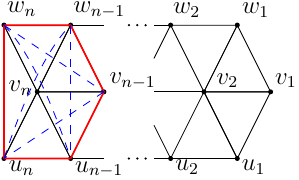}
\end{center}
\caption{The $5$-cycle $\langle w_n,w_{n-1},v_{n-1},u_{n-1},u_n\rangle$ in red. One noncrossing pair of blue edges belongs to $Y$.}\label{fig:dlcp1sys}
\end{figure}

Since neither 1) nor 2) holds, we obtain a shorter double ladder with the cap $[u_{n-1},v_{n-1},w_{n-1}]$.
Since there is no $i$ such that 1) or 2) holds, by induction we obtain that $(u_1, w_{1})\in Y$. Contradiction.
\end{proof}

\section{Curvature}\label{cv}
Let $D$ be a disc diagram that considered as a CW-complex is a square complex. In this section we will call such diagrams \textit{square disc diagrams}.

The \textit{square curvature} of a vertex $v\in D$ is defined as
$$
\kappa^{\square}_D(v)=4 - 2\delta_D(v) + \rho^{\square}_D(v),
$$
where $\delta_D(v)$, as previously, denotes the valence of vertex $v$ and $\rho^{\square}_D(v)$ denotes the number of squares incident to $v$, cf. Fig.\ref{fig:class}.

\begin{figure}[H]
\begin{center}
\includegraphics[scale=1.5]{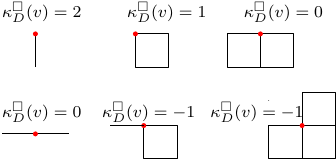}
\end{center}
\caption{All the possible neighborhoods of a boundary vertex $v$ (red) of a square disc diagram $D$ such that $-1\leq\kappa^{\square}_D(v)\leq 2$.}\label{fig:class}
\end{figure}

Analogously we define the curvature for simplicial complexes. Let $D$ be a disc diagram that considered as a CW-complex is a simplicial complex. In this section, we will call such diagrams \textit{simplicial disc diagrams}.

The \textit{simplicial curvature} of a vertex $v\in D$ is defined as follows
$$
\kappa^{\triangle}_D(v)=6 - 3\delta_D(v) + 2\rho^{\triangle}_D(v),
$$
where $\rho^{\triangle}(v)$ denotes the number of triangles incident to $v$, cf. Fig.\ref{fig:class2}.

\begin{figure}[H]
\begin{center}
\includegraphics[scale=1.5]{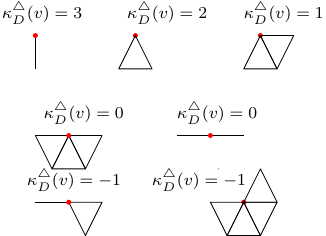}
\end{center}
\caption{All the possible neighborhoods of a boundary vertex $v$ (red) of a simplicial disc diagram $D$ such that $-1\leq\kappa^{\triangle}_D(v)\leq 3$.}\label{fig:class2}
\end{figure}

We now state a version of the combinatorial Gauss-Bonnet theorem for CAT(0) square and simplicial disc diagrams. Generalization of the Gauss-Bonnet theorem to arbitrary $2$-complexes was stated and proven by McCammond and Wise in \cite[Theorem 4.6]{MW}. The version for CAT(0) square disc diagrams was already stated by Hoda in \cite[Proposition 1.8]{hoda2019quadric}.

\begin{pr}[Gauss-Bonnet Theorem for a CAT(0) Disc Diagrams]\label{gb}
Let $D$ be a CAT(0) square disc diagram and let $E$ be a CAT(0) simplicial disc diagram. Then: 
$$
\sum\limits_{v \in D}\kappa^{\square}_D(v)=4 
\text{ and }
\sum\limits_{v \in E}\kappa^{\triangle}_E(v)=6, 
$$
moreover:
$$
\sum\limits_{v \in \partial D}\kappa^{\square}_D(v)\geq 4
\text{ and }
\sum\limits_{v \in \partial E}\kappa^{\triangle}_E(v)\geq 6. 
$$
\end{pr}
\begin{proof}
We only show the proof in the case of simplicial disc diagrams, case of square disc diagrams is analogous, and was already proved by Hoda.

For a disc diagram the Euler characteristic $\chi(E)$ is equal to $1$. It can be computed by subtracting the number of edges from the number of vertices and triangles. That is, each edge contributes $-1$ to the Euler characteristic and each vertex or $2$-cell contributes $+1$.
Distributing $\frac{-1}{2}$ to each of the endpoints of the edge, and $+\frac{1}{3}$ to each of vertices at the boundary of each triangle we obtain:

$$
\chi(E)=\sum\limits_{v \in E}\frac{\kappa^{\triangle}_E(v)}{6}.
$$

In the case of internal vertices, we have $\delta_E(v)=\rho^{\triangle}_E(v)$ and therefore by the CAT(0) property $\kappa^{\triangle}_E(v)$ is nonpositive for the internal vertices.
\end{proof}

The square and the simplicial curvatures have analogous properties, from now on we will often not differentiate between them and just denote it by $\kappa_D(v)$ with the type of the curvature known from context.

\begin{pr}\label{geo}
Let $D$ be either square or simplicial disc diagram and let $\gamma\subset\partial D$ be a geodesic in $D$. Then none of the internal vertices of $\gamma$ has the curvature greater than $1$. Moreover, if $u,v$ are internal vertices of $\gamma$ with $\kappa_D(u)=\kappa_D(v)=1$, then there is a vertex $w\in \gamma$ between $v$ and $u$ with $\kappa_D(u)\leq -1$.
\end{pr}

\begin{proof}
Let $\gamma:=(v_0,v_1,\ldots,v_n)$. If $\kappa^{\square}_D(v_i)>1$, then, by Figure \ref{fig:class}, we have $v_{i-1}=v_{i+1}$. If $\kappa^{\triangle}_D(v_i)>1$, then, by Figure \ref{fig:class2}, either $v_{i-1}=v_{i+1}$ or there exists an edge between $v_{i-1}$ and $v_{i+1}$. Either case contradicts the fact that $\gamma$ is a geodesic.

Let $u=v_i$ and  $v=v_j$. Assume, without loss of generality, that $\kappa_D(v_s)=0$ for $i<s<j$. 
If $D$ is a square disc diagram, the distance between $v_{i-1}$ and $v_{j+1}$ is equal to at most the distance between $v_i$ and $v_j$.  
If $D$ is a simplicial disc diagram, the distance between
$v_{i-1}$ and $v_{j+1}$ is equal to at most the distance between $v_i$ and $v_j$ increased by one.
Either case contradicts the fact that $\gamma$ is a geodesic.
\end{proof}

In fact, these properties of geodesics stated above characterize geodesics in systolic and quadric complexes. 

\begin{lm}\label{geo2}
Let $\gamma:=\{g_0,g_1,\ldots,g_n\}$ be a path in an either systolic or quadric complex $X$ and let $\gamma':=\{g_0'=g_0,g_1',\ldots,g'_m=g_n\}$ be a geodesic between the endpoints of $\gamma$ such that a minimal area disc diagram $D$ for $\gamma\cup\gamma'$ in $X$ has the smallest area.
If for all $0<i\leq j<n$ we have
$$\sum\limits_{i\leq k\leq j}\kappa_D(g_k)\leq 1,$$
then $\gamma$ is a geodesic.
\end{lm}

\begin{proof}
By Proposition \ref{geo} all vertices along $\gamma'$ have nonpositive curvature in $D$. Indeed, if for some $0<i<m$ we have $\kappa_D(g'_i)=1$, then there exists a vertex $g_i''\in D$ incident to $g'_{i-1}$ and $g'_{i+1}$ (see Figures \ref{fig:class} and \ref{fig:class2}). 
Therefore, there exists a smaller area disc diagram $D'$ for $\gamma\cup\gamma''$, where $\gamma'':=\{g'_0,\ldots,g'_{i-1},g_i'',g'_{i+1}\ldots, g'_n\}$. 
Therefore, the sum of the curvatures along $\gamma'$ is bounded by $0$. 

On the other hand, the curvature along $\gamma$ is bounded by $1$. 
Then by nonpositiveness of the curvature of the internal vertices and Proposition \ref{gb}, it follows that the sum of the curvature of vertices $g_0,g_n$ is at least $5$ in the case of the simplicial disc diagrams and at least $3$ in the case of square disc diagrams. In either case at least one of $g_0, g_n$ has to be a spur, i.e. a vertex of valence $1$.

We will show that $D$ is a tree.
Assume that $D$ is not a tree and without loss of generality assume that $g_0$ is a spur. We take the smallest $i$ such that $g_i$ is on the boundary of a $2$-cell.
Then the curvature of $g_i=g_i'$ has to be negative (see Figures \ref{fig:class}, \ref{fig:class2}), and the curvature of vertices $g_j=g_j'$ for $0<j<i$ is $0$. Therefore, the sum of the curvatures along $\gamma'$ is negative. 
The sum of the curvatures of vertices $g_j$ for $i<j<n$ is bounded by $1$. Therefore the sum of the curvature along $\gamma$ is $0$. 
Thus the sum of the curvatures of vertices $g_0,g_n$ is at least $7$ in the case of the simplicial disc diagrams, or at least $5$ in the case of square disc diagrams. This is possible only if one of these vertices is not connected to the rest of the diagram, a contradiction.
 
If $D$ is a tree and $\gamma$ is not a geodesic, then there exists vertex $g_j$ for $0<j<n$ that is a spur, but spurs have curvature greater than $1$, a contradiction.
\end{proof}

\section{Combining geodesics}
\label{negcuv}

In this section, in Lemma~\ref{lad} and Lemma~\ref{ladc6}, we show that some particular paths obtained via gluing two geodesics at their endpoints are geodesics themselves. This is used in Section~\ref{inford} to construct automorphisms of infinite order.

\subsection{\cftfs case.}
\label{negcuvcftf}

Here we assume that $X$ is a simply connected \cftfs small cancellation complex such that every $1$-cell of $X$ is contained in the boundary of a $2$-cell. Let $Y$ be the quadrization of $X$. Let $G$ be a finitely generated group acting on $X$ by automorphisms and assume that this action induces a free action on the $1$-skeleton $X^{1}$ of $X$. It is clear that $G$ acts on $Y$ by automorphisms  and this action induces a free action on the sets of edges and vertices from $X_0$. Moreover, since by Proposition \ref{ctc} any non-empty intersection of two $2$-cells from $X$ is connected, the action of $G$ on $Y$ induces a free action on the set of squares as well.
From now on, if $D$ is a diagram in $Y$ and $v$ is a vertex in $D$ then $v$ is mapped to a vertex in $Y$ denoted by $v^Y$. In some cases, we denote by $u^D$ a vertex in $D$ that is mapped to a vertex $u\in Y$. Note that there is some ambiguity, as more than one vertex can be mapped to $u$, therefore we only use this notation when this ambiguity does not matter.  

The aim of this section is to prove a technical Lemma \ref{lad}, which is necessary for the proof of the \cftfs part of Lemma \ref{inf}, the main lemma of Section \ref{inford}. But first, we need to prove the following.

\begin{lm}\label{laddy}
Let $D$ be a minimal area disc diagram in $Y$ and $x\in D$ be a vertex such that $x^Y$ is fixed by some $h\in G$.
Let $(u_{n}, u_{n-1},\ldots, u_1, u_0=x=v_0, v_1,\ldots, v_{n},v_{n+1}, \overline{u}_{n}, \overline{u}_{n-1},\ldots, \overline{u}_1, \overline{u}_0=\overline{x}=\overline{v}_0, \overline{v}_1,\ldots, \overline{v}_{n},\overline{v}_{n+1})$ be a tuple of vertices such that (see Figure \ref{fig:llad}):
\begin{enumerate}[1)]
\item $hu^Y_i=v^Y_i$ for $0\leq i\leq n$;
\item $h\overline{u}^Y_i\neq\overline{v}^Y_i$ for $0\leq i\leq n$;
\item $\{u_n, u_{n-1},\ldots, u_1, x, v_1,\ldots, v_{n},v_{n+1}| \overline{u}_n, \overline{u}_{n-1},\ldots, \overline{u}_1, \overline{x}, \overline{v}_1,\ldots, \overline{v}_{n},\overline{v}_{n+1}\}$ is a ladder in $D$;
\item  $\{h\overline{x}^Y,h\overline{u}_1^Y\ldots, h\overline{u}_{n}^Y|x^Y,v_1^Y,\ldots, v_{n}^Y| \overline{x}^Y,\overline{v}_1^Y,\ldots, \overline{v}_{n}^Y|\overline{v}_{n+1}^Y\}$ is a double ladder with a cap in $Y$.
\end{enumerate}
\begin{figure}[H]
\begin{center}
\includegraphics[scale=0.9]{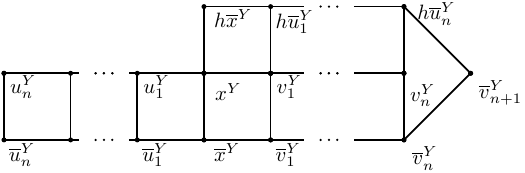}
\end{center}
\caption{}\label{fig:llad}
\end{figure}

Then at least one of the following holds: 
\begin{enumerate}[(i)]
\item $(\overline{v}_{s}^Y, v_{s-2}^Y)\in Y$ for some $2\leq s \leq n+1$;
\item $(\overline{u}_{s}^Y, u_{s-2}^Y)\in Y$ for some $2\leq s \leq n$ ; 
\item $d_{Y_x}(h\overline{x}^Y,\overline{x}^Y)< 2$.
\end{enumerate} 
\end{lm}

\begin{proof}
By Proposition \ref{dlwc} at least one of the following 
holds:
\begin{enumerate}[(a)]
\item $(\overline{v}_{n+1}^Y, v_{n-1}^Y)\in Y$;
\item $(h\overline{u}_{s}^Y, v_{s-2}^Y)\in Y$ for some $2 \leq s\leq n$;
\item $(\overline{v}_{s}^Y, v_{s-2}^Y)\in Y$ for some $2 \leq s\leq n$;
\item $(h\overline{u}_{1}^Y, \overline{x}^Y)\in Y$; 
\item $(\overline{v}_{1}^Y, h\overline{x}^Y)\in Y$.
\end{enumerate}
The case (i) is satisfied if (a) or (c) holds. 

If (b) holds, then case (ii) is satisfied. Indeed, if for some $2\leq s\leq n$ we have $(h\overline{u}_{s}^Y, v_{s-2}^Y=hu_{s-2}^Y)\in Y$, then $(\overline{u}_{s}^Y, u_{s-2}^Y)\in Y$. 

The case (iii) is satisfied if either (d) or (e) holds. Indeed, if one of these cases holds, then $R_1:=[\overline{x}^Y, x^Y, h\overline{x}^Y,h\overline{u}_{1}^Y]$ or $R_2:=[\overline{x}^Y, x^Y, h\overline{x}^Y,\overline{v}_{1}^Y]$ is a square in $Y$, therefore $d_{Y_x}(h\overline{x}^Y,\overline{x}^Y)<2$.
\end{proof}

\begin{lm}\label{lad}
Let $x\in X_2$ and $g \in G$ be such that $x\in \mathrm{Fix}_Y(g)$. Assume that for every square $P\in Y$ containing $x$ in its boundary, we have $P\cap gP= \{x\}$.
Let $\gamma_1 := (x=x_0, x_1,\ldots, x_n)$ and $\gamma_2 :=(x=y_0, y_1,\ldots, y_m)$ be a geodesics in $Y$, such that $n\leq m$ and for all $i\leq n$ we have $gx_i=y_i$.
Then $\alpha:=\gamma_1\cup \gamma_2$ is a geodesic.
\end{lm}
\begin{proof}
Assume that $\alpha$ is not a geodesic. Then there exists a geodesic $\beta$ between $x_n$ and $y_m$.
From the set of geodesics between $x_n$ and $y_m$ choose $\beta$ such that the minimal area disc diagram $D$ for the path $\alpha\cup\beta$ has the smallest area. 

We want to show that for any $u,v\in \alpha$ such that $x_n<u^Y<v^Y<y_m$, if we have $\kappa_D(u)=\kappa_D(v)=1$, then there is a vertex $w\in \alpha$, such that
$u^Y<w^Y<v^Y$ and $\kappa_D(w)\leq -1$. 
It is clear that in such case the sum of the curvature along any subpath of $\alpha$ is bounded by $1$, and the assumptions of Lemma \ref{geo2} are satisfied.

First, observe that between $u$ and $v$ there is no vertex with curvature $2$. 
Indeed, a boundary vertex has curvature $2$ iff it is an end of a spur (see Figure \ref{fig:class}). By Proposition \ref{geo} such a vertex cannot be an internal vertex of a geodesic, so it has to be mapped to $x$. But if $x^D$ is a spur then $g$ fixes a vertex from $X_0$, a contradiction. 
Therefore each vertex between $u$ and $v$ is incident to at least one square (see Figure \ref{fig:class}).

Assume that there are no vertices of curvature at most $-1$ between $u$ and $v$.
Without loss of generality, we can assume that all vertices between $u$ and $v$ have curvature $0$ and by Figure \ref{fig:class} each of them is incident to two squares.
Observe, that $u^Y\neq x \neq v^Y$.
Indeed, if $\kappa_D(x^D)=1$ then $x^D$ is incident to exactly one square $P\in D$. Therefore there is $P^Y\in Y$ corresponding to $P$ such that $P^Y\cap gP^Y\neq \{x\}$, a contradiction. 

By Proposition \ref{geo}, $u$ and $v$ cannot be internal vertices of the same geodesic. Therefore we can assume that $u^Y\in\gamma_1$, $v^Y\in\gamma_2$.

Now we will find a tuple satisfying the assumptions 1)-4) of Lemma \ref{laddy}.
We begin by taking the sequence $u_N=u, u_{N-1},\ldots, u_1, u_0=x^D=v_0, v_1,\ldots, v_{M-1},v_M=v$ of consecutive vertices in $\partial D$ between $u$ and $v$. We have $gu^Y_p=v^Y_p$ for $p\leq N$.

We assumed that $u_N,v_M$ have curvature $1$ and all of vertices between them have curvature $0$. Because of that the neighborhood of $u_N,\ldots, u_1, x^Ds, v_1,\ldots,v_M$ has the following form.
The vertex $x^D$ is adjacent to three vertices, two of them being $u_{1}$, $v_{1}$, we denote the remaining adjacent vertex by $\overline{x}$.
The vertex $u_N$ (resp. $v_M$) is adjacent to two vertices, one of them $u_{N-1}$ (resp. $v_{M-1}$), we denote the remaining adjacent vertex by $\overline{u_N}$ (resp. $\overline{v_M}$). 
For $0<s<N$ (resp. $0<p<M$) the vertex $u_s$ (resp. $v_p$) is adjacent to three vertices, two of them being $u_{s-1}$, $u_{s+1}$ (resp. $v_{p-1}$, $v_{p+1}$).  We denote the remaining adjacent vertex by $\overline{u}_s$ (resp. $\overline{v}_p$), see Figure \ref{fig:lla1}.

\begin{figure}[H]
\begin{center}
\includegraphics[scale=0.9]{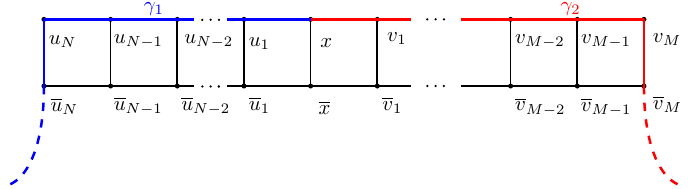}
\end{center}
\caption{Disc diagram $D$ in the neighborhood of $u_N,\ldots, u_1, x, v_1,\ldots,v_M$.}\label{fig:lla1}
\end{figure}

Let the sequence $\overline{u}_N^Y, \overline{u}_{N-1}^Y,\ldots, \overline{u}_1^Y,\overline{u}_0^Y=\overline{x}^Y=\overline{v}_0^Y, \overline{v}_1^Y,\ldots, \overline{v}^Y_{M-1},\overline{v}^Y_M$, be the corresponding sequence of vertices in $Y$. We have $\overline{u_N}^Y\in \gamma_1$ and $\overline{v_M}^Y\in \gamma_2$.

For any $0<p\leq M$, $0<s\leq N$ we define squares $V_p=[v^Y_{p-1},v^Y_p,\overline{v}^Y_p,\overline{v}^Y_{p-1}]$ and $U_s=[u_{s-1}^Y,u_s^Y,\overline{u}_s^Y,\overline{u}_{s-1}^Y]$ (we remind that $x=u_0^Y=v_0^Y$ and $\overline{x}^Y=\overline{u}_0^Y=\overline{v}_0^Y$). We consider the subcomplex $S\subset Y$ consisting of $gU_1, \ldots gU_N,V_1,\ldots V_M$.
Let us remind that for $s\leq N$ we have $g u^Y_s=v^Y_s$. We note here that it is possible that $g\overline{u}^Y_s=\overline{v}^Y_s$ as shown in Figure \ref{fig:lla2}.
\begin{figure}[H]
\begin{center}
\includegraphics[scale=0.9]{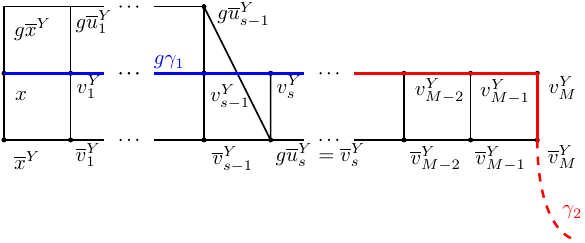}
\end{center}
\caption{Example of the subcomplex $S$. In this case $g\overline{u}^Y_s=\overline{v}^Y_s$ and $g\overline{u}^Y_{s-1}\neq\overline{v}^Y_{s-1}$.}\label{fig:lla2}
\end{figure}
It is clear that one of the following cases holds.

\begin{enumerate}[1)]

\item $g\overline{x}^Y=\overline{x}^Y$.

\item $g\overline{x}^Y\neq\overline{x}^Y$, and there is some $p\leq N$, such that $g\overline{u}^Y_p= \overline{v}^Y_p$ and $g\overline{u}^Y_s\neq \overline{v}^Y_s$ for $s < p$. If $p=N$ then we have $N=M$ as $\overline{v}^Y_N=g\overline{u}^Y_N \in \alpha$.

\item $g\overline{x}^Y\neq\overline{x}^Y$, $gu^Y\neq v^Y$ and for $s \leq N$ we have $g\overline{u}^Y_s\neq \overline{v}^Y_s$. In this case we have $g\overline{u}^Y_N=v^Y_{N+1}$.
\end{enumerate}

We will show that each of these cases results in a contradiction. 

In case 1) $g$ is not an identity and fixes $\overline{x}^Y$. But $\widehat{\overline{x}^Y}$ is a $0$-cell in $X$, therefore we have a contradiction to the freeness of the action of $G$ on $X^1$.

In case 2) we have $g\overline{u}^Y_p= \overline{v}^Y_p$ and $gu^Y_{p-1}\neq v^Y_{p-1}$. It follows that $\partial (gU_p\cup V_p) = \langle\overline{v}^Y_{p},\overline{v}^Y_{p-1},v^Y_{p-1},g\overline{u}^Y_{p-1}\rangle$ is a $4$-cycle in $Y$. Since $Y$ is the quadrization of the complex $X$, each $4$-cycle spans a square and $[\overline{v}^Y_{p},\overline{v}^Y_{p-1},v^Y_{p-1},g\overline{u}^Y_{p-1}]\in Y$.
Figure \ref{fig:ll1} show the complex $S$.
\begin{figure}[H]
\begin{center}
\includegraphics[scale=0.9]{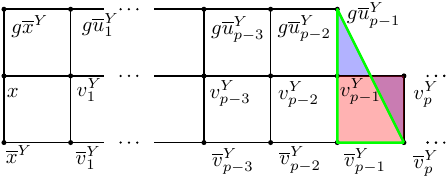}
\end{center}
\caption{The boundary of the union of $gU_p$ (blue+purple) and $V_p$ (red+purple) is a $4$-cycle (green).}\label{fig:ll1}
\end{figure}
Thus the cycle $\langle \overline{v}^Y_{p},\overline{v}^Y_{p-1},\ldots ,\overline{v}^Y_{1},\overline{x}^Y, x,g\overline{x}^Y, g\overline{u}^Y_{1},\ldots ,g\overline{u}^Y_{p-1}\rangle $ bounds a double ladder with a cap, where the square $[\overline{v}^Y_p,\overline{v}^Y_{p-1},v^Y_{p-1},g\overline{u}^Y_{p-1}]$ is a cap. 
Clearly, the tuple 
$$
(u_{p-1}, u_{p-2},\ldots, u_1, x, v_1,\ldots, v_{p-1},v_{p}, \overline{u}_{p-1}, \overline{u}_{p-2},\ldots, \overline{u}_1, \overline{x}, \overline{v}_1,\ldots, \overline{v}_{p-1},\overline{v}_p)
$$
satisfies the conditions of Lemma \ref{laddy}.
Therefore one of the following holds:
\begin{enumerate}[(i)]
\item for some $2\leq s \leq p$, $(\overline{v}_{s}^Y, v_{s-2}^Y)\in Y$;
\item for some $2\leq s \leq p-1$, $(\overline{u}_{s}^Y, u_{s-2}^Y)\in Y$; 
\item there exists a square $P$ in $Y$ such that $x\in P$ and $P\cap gP\supsetneq \{x\}$.
\end{enumerate}

In case (i) both $\overline{v}_M$ and $v_{s-2}$ are vertices belonging to the geodesic $\gamma_2$. Therefore $(v_{s-2},\ldots, v_M, \overline{v}_M)$ is a geodesic, but $(v_{s-2},\overline{v}_{s},\ldots, \overline{v}_M)$ is a shorter path between the same pair of vertices, a contradiction.

In case (ii) both $\overline{u}_N$ and $u_{s-2}$ are vertices belonging to the geodesic $\gamma_1$. Therefore $(u_{s-2},\ldots, u_N, \overline{u}_N)$ is a geodesic, but $(u_{s-2},\overline{u}_{s},\ldots, \overline{u}_N)$ is a shorter path between the same pair of vertices, a contradiction.

If (iii) holds, then we have a square $P\ni x$ such that $P\cap gP\neq \{x\}$. A contradiction.

In case 3) we have a double ladder with a cap in $Y$ bounded by the cycle 
$$\langle\overline{v}^Y_{N+1},\overline{v}^Y_{N},\ldots, \overline{v}^Y_{1},\overline{x}^Y, x,g\overline{x}^Y, g\overline{u}^Y_{1},\ldots ,g\overline{u}^Y_N\rangle$$ (see Figure \ref{fig:ll5}). 
\begin{figure}[H]
\begin{center}
\includegraphics[scale=0.9]{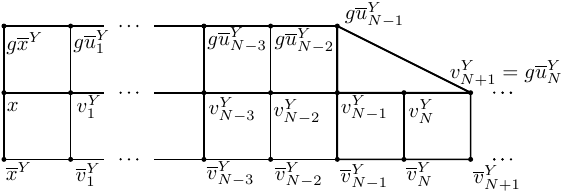}
\end{center}
\caption{Double ladder with a cap in $Y$. The square $[v^Y_{N+1},\overline{v}^Y_{N+1},\overline{v}^Y_{N},v^Y_{N}]$ is a cap.}\label{fig:ll5}
\end{figure}

Clearly, the set 
$$
\{u_N, u_{N-1},\ldots, u_1, x, v_1,\ldots, v_{N},v_{N+1}, \overline{u}_N, \overline{u}_{N-1},\ldots, \overline{u}_1, \overline{x}, \overline{v}_1,\ldots, \overline{v}_{N},\overline{v}_{N+1}\}
$$
satisfies the conditions of Lemma \ref{laddy}.
Again each case following from this lemma results in a contradiction.

Therefore between any two vertices $x_n<u<v<y_m$ such that $\kappa_D(u)=\kappa_D(v)=1$, there is a vertex $w$ such that $u^Y<w^Y<v^Y$ and the curvature of $w$ is at most $-1$. 
As already mentioned, in such a case $\alpha$ satisfies the conditions of Lemma \ref{geo2}, thus $\alpha$ is a geodesic.
\end{proof}

\subsection{\css case.}\label{negcuvcs}
This subsection is quite similar to the previous one, with lemmas and proofs in this section being analogous to the ones from the previous section. We begin this section with analogous assumptions.

We assume that $X$ is a simply connected \css small cancellation complex such that every $1$-cell of $X$ is contained in the boundary of a $2$-cell. Let $Y$ be the Wise complex of $X$. Let $G$ be a finitely generated group acting on $X$ by automorphisms and assume that this action induces a free action on the $1$-skeleton $X^{1}$ of $X$. It is clear that $G$ acts on $Y$ by automorphisms. From now on, if $D$ is a diagram in $Y$ and $v$ is a vertex in $D$ then $v$ is mapped to a vertex in $Y$ denoted by $v^Y$. 

The aim of this section is to prove a technical Lemma \ref{ladc6}, which is a \css analogue of Lemma \ref{lad} and is necessary for the proof of the \css part of Lemma \ref{inf}, the main lemma of the Section \ref{inford}. We begin with a \css analogue of Lemma \ref{laddy}. 

\begin{lm}\label{laddyc6}
Let $D$ be a minimal area disc diagram in $Y$ and $x\in D$ be a vertex such that $x^Y$ is fixed by some $h\in G$.

Let $(u_{n}, u_{n-1},\ldots, u_1, u_0=x=v_0, v_1,\ldots, v_{n}, \overline{u}_{n}, \overline{u}_{n-1},\ldots, \overline{u}_1,\overline{u}_0, \overline{v}_0, \overline{v}_1,\ldots, \overline{v}_{n-1},\overline{v}_{n})$ be a tuple of vertices such that (see Figure \ref{fig:lladcs}):
\begin{enumerate}[1)]
\item $hu^Y_i=v^Y_i$ for $0\leq i\leq n$;
\item $h\overline{u}^Y_i\neq\overline{v}^Y_i$ for $0\leq i\leq n$;
\item $\{u_n, u_{n-1},\ldots, u_1, x, v_1,\ldots, v_{n}| \overline{u}_{n+1}, \overline{u}_{n},\ldots, \overline{u}_1,\overline{u}_0, \overline{v}_0, \overline{v}_1,\ldots, \overline{v}_{n},\overline{v}_{n+1}\}$ is a ladder in $D$;
\item  $\{h\overline{u}_0^Y,h\overline{u}_1^Y\ldots, h\overline{u}_{n}^Y|x^Y,v_1^Y\ldots, v_{n}^Y|\overline{v}_0^Y,\ldots, \overline{v}_{n}^Y\}$ is a double ladder with a cap in $Y$.
\end{enumerate}

Then at least one of the following holds: 
\begin{enumerate}[(i)]
\item for some $1\leq s \leq n$, $(\overline{v}_{s}^Y, v_{s-1}^Y)\in Y$;
\item for some $1\leq s \leq n$, $(\overline{u}_{s}^Y, u_{s-1}^Y)\in Y$; 
\item $d_{Y_x}(h\overline{u}^Y_0,\overline{u}^Y_0)< 3$.
\end{enumerate} 
\end{lm}
\begin{figure}[H]
\begin{center}
\includegraphics[scale=0.9]{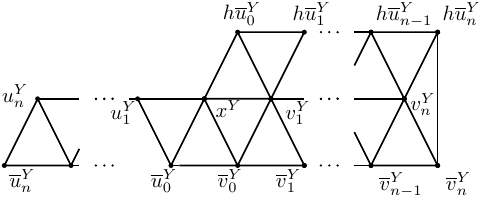}
\end{center}
\caption{}\label{fig:lladcs}
\end{figure}
\begin{proof}
By Proposition \ref{dlwcsys} at least one of the following 
holds:
\begin{enumerate}[(a)]
\item for some $1 \leq s\leq n$ $(\overline{v}_{s}^Y, v_{s-1}^Y)\in Y$;
\item for some $1 \leq s\leq n$ $(h\overline{u}_{s}^Y, v_{s-1}^Y)\in Y$;
\item $(h\overline{u}_{0}^Y, \overline{v}_0^Y)\in Y$; 
\end{enumerate}
The case (i) is satisfied if (a) holds.

The case (ii) is satisfied if (b) holds. Indeed, if for some $1\leq s\leq n+1$ we have $(h\overline{u}_{s}^Y, v_{s-1}^Y=hu_{s-1}^Y)\in Y$, then $(\overline{u}_{s}^Y, u_{s-1}^Y)\in Y$. 

The case (iii) is satisfied if (c) holds. Indeed, $d_{Y_x}(h\overline{u}_{0}^Y,\overline{u}_{0}^Y)\leq 2$, follows from $\overline{u}_{0}^Y$ being incident to $\overline{v}_{0}^Y$.
\end{proof}

\begin{lm}\label{ladc6}
Let $x\in Y^0$ and $g \in G$, such that $x\in \mathrm{Fix}_Y(g)$. Assume that for every vertex $y\in Y$ belonging to the link of $x$ in $Y$, we have $d_{Y_x}(gy,y)\geq 3$.
Let $\gamma_1 := (x_0=x, x_1,\ldots, x_n)$ and $\gamma_2 :=(y_0=x, y_1,\ldots, y_m)$ be a geodesics in $Y$, such that $n\leq m$ and for all $i\leq n$ we have $gx_i=y_i$.
Then $\alpha:=\gamma_1\cup \gamma_2$ is a geodesic.
\end{lm}

\begin{proof}
The proof is analogous to the proof of Lemma \ref{lad} and has exactly the same structure.
We take a geodesic $\beta$ between $x_n$ and $y_m$ such that minimal area disc diagram $D$ for the path $\alpha\cup\beta$ has the smallest area. 

As previously, we want to use Lemma \ref{geo2}.
Therefore we need to show that assumptions of that Lemma are satisfied, and again we do it by showing that for any pair of vertices $u,v\in \alpha$, such that $x_n<u^Y<v^Y<y_m$, if we have $\kappa_D(u)=\kappa_D(v)=1$, then there is a vertex $w\in \alpha$, such that
$u^Y<w^Y<v^Y$ and $\kappa_D(w)\leq -1$.

First, observe that between $u$ and $v$ there is no vertex with curvature greater than $1$. A boundary vertex has curvature $3$ iff it is an end of the spur and $2$ iff it belongs to the boundary of exactly one triangle (see Figure \ref{fig:class2}).  By Proposition \ref{geo} such a vertex cannot be an internal vertex of a geodesic, so it has to be mapped to $x$. But for every vertex $y\in Y$ belonging to the link of $x$ in $Y$, we have $d_{Y_x}(gy,y)\geq 3$. Therefore, the curvature of $x$ is at most $0$.
Assume that there are no vertices of the curvature at most $-1$ between $u$ and $v$.
Then each vertex between $u$ and $v$ is incident to at most three triangles (see Figure \ref{fig:class2}).
We can assume that all vertices between $u$ and $v$ have curvature $0$ and by Figure \ref{fig:class2} each of them is incident to exactly three triangles.

By Proposition \ref{geo}, $u$ and $v$ cannot be internal vertices of the same geodesic. Therefore, we can assume that $u\in\gamma_1$, $v\in\gamma_2$. 

Analogously to the \cftfs version of this Lemma, we now want to find a tuple satisfying  the assumptions 1)-4) of Lemma \ref{laddyc6}. Therefore we take the sequence $u_N=u, u_{N-1},\ldots, u_1, u_0=x^D=v_0, v_1,\ldots, v_{M-1},v_M=v$ of consecutive vertices in $\partial D$ between $u$ and $v$.
Since $N<n$ we have $gu^Y_p=v^Y_p$ for $p\leq N$.

We assumed that $u_N,v_M$ have curvature $1$ and all of vertices between them have curvature $0$. Because of that the neighborhood of $u_N,\ldots, u_1, x, v_1,\ldots,v_M$ has the following form.
Vertex $x$ is adjacent to four vertices, two of them being $u_{1}$, $v_{1}$, we denote the remaining adjacent vertices by $\overline{u}_0,\overline{v}_0$ so that $(u_1,\overline{u}_0),(v_1,\overline{v}_0)\in D$.
For $0<s<N$ (resp. $0<p<M$) the vertex $u_s$ (resp. $v_p$) is adjacent to four vertices, two of them being $u_{s-1}$, $u_{s+1}$ (resp. $v_{p-1}$, $v_{p+1}$). One of the remaining two vertices is connected to $u_{s-1}$ (resp. $v_{p-1}$), we denote it by $\overline{u}_{s-1}$ (resp. $\overline{v}_{p-1}$), the last vertex is denoted by $\overline{u}_s$ (resp. $\overline{v}_p$).
Vertex $u_N$ (resp. $v_M$) is adjacent to three vertices, two of them being $u_{N-1}$, $\widehat{u}_{N-1}$ (resp. $v_{M-1}$, $\widehat{v}_{M-1}$), we denote the remaining adjacent vertex by $\overline{u}_{N}$ (resp. $\overline{v}_{M}$), see Figure \ref{fig:lla1sys}.

\begin{figure}[H]
\begin{center}
\includegraphics[scale=0.9]{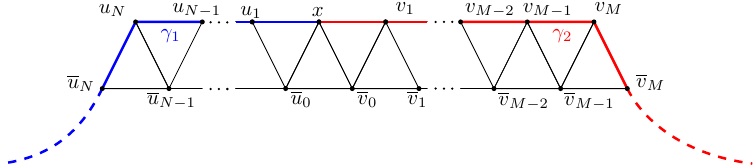}
\end{center}
\caption{The disc diagram $D$ in the neighborhood of $u_N,\ldots, u_1, x, v_1,\ldots,v_M$.}\label{fig:lla1sys}
\end{figure}

Let the sequence $\overline{u}_N^Y, \overline{u}_{N-1}^Y,\ldots, \overline{u}_1^Y,\overline{u}_0^Y,\overline{v}_0^Y, \overline{v}_1^Y,\ldots, \overline{v}^Y_{M-1},\overline{v}^Y_M$, be the corresponding sequence of vertices in $Y$. We have $\overline{u}_N^Y\in \gamma_1$ and $\overline{v}_M^Y\in \gamma_2$.

For any $0<p\leq M$, $0<s\leq N$ we define the triangles $V_p=[v^Y_{p-1},v^Y_p,\overline{v}^Y_{p-1}]$, $\overline{V}_p=[v^Y_p,\overline{v}^Y_p,\overline{v}^Y_{p-1}]$, $U_s=[u_{s-1}^Y,u_s^Y,\overline{u}_{s-1}^Y]$, and $\overline{U}_s=[u_s^Y,\overline{u}_s^Y,\overline{u}_{s-1}^Y]$ (we remind that $x=u_0^Y=v_0^Y$). We consider the subcomplex $S\subset Y$ consisting of $gU_1, \ldots, gU_N,g\overline{U}_1, \ldots, g\overline{U}_N, V_1,\ldots, V_M,\overline{V}_1,\ldots, \overline{V}_M$.
Let us remind that for $s\leq N$ we have $g u^Y_s=v^Y_s$. Similarly as in the \cftfs case, it is possible that $g\overline{u}^Y_s=\overline{v}^Y_s$, as shown in Figure \ref{fig:lla2sys}.
\begin{figure}[H]
\begin{center}
\includegraphics[scale=0.9]{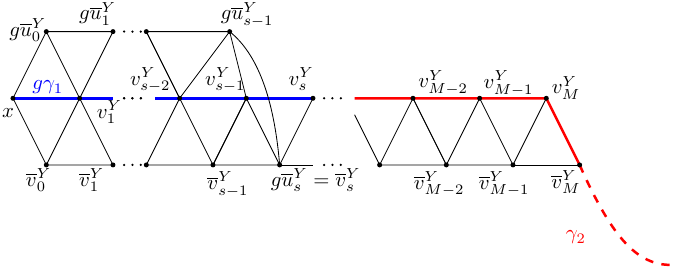}
\end{center}
\caption{Example of the subcomplex $S$. In this case $g\overline{u}^Y_s=\overline{v}^Y_s$ and $g\overline{u}^Y_{s-1}\neq\overline{v}^Y_{s-1}$.}\label{fig:lla2sys}
\end{figure}
It is clear that one of the following cases holds.

\begin{enumerate}[1)]

\item $g\overline{u}^Y_0=\overline{v}_0^Y$.

\item $g\overline{u}^Y_0\neq\overline{v}_0^Y$, and there is some $p\leq N$, such that $g\overline{u}^Y_p= \overline{v}^Y_p$ and $\overline{v}^Y_{s+1}\neq g\overline{u}^Y_s\neq \overline{v}^Y_s$ for $s < p$. If $p=N$ then we have $N=M$ as $\overline{v}^Y_N=g\overline{u}^Y_N \in \alpha$.

\item $g\overline{u}^Y_0\neq\overline{v}_0^Y$, and there is some $p\leq N$, such that $g\overline{u}^Y_p= \overline{v}^Y_{p+1}$ and $\overline{v}^Y_{s+1}\neq g\overline{u}^Y_s\neq \overline{v}^Y_s$ for $s < p$.

\item $g\overline{u}^Y_0\neq\overline{v}_0^Y$, $gu^Y\neq v^Y$ and for $s \leq N$ we have $\overline{v}^Y_{s+1}\neq g\overline{u}^Y_s\neq \overline{v}^Y_s$. In this case we have $g\overline{u}^Y_N=v^Y_{N+1}$.
\end{enumerate}

Like in the proof of Lemma \ref{lad} we will show that each of these cases results in a contradiction. 

In case 1)  $\overline{u}^Y_0$ belongs to the link of $x$ in $Y$ and $d_{Y_x}(g\overline{u}^Y_0,\overline{u}^Y_0)<3$, a contradiction.

In case 2) we have $g\overline{u}^Y_p= \overline{v}^Y_p$ and $gu^Y_{p-1}\neq v^Y_{p-1}$. It follows that there is a cycle of the length $4$ in the link of $v^Y_{p}$.
Figure \ref{fig:ll1sys} shows the complex $S$ in that case.
\begin{figure}[H]
\begin{center}
\includegraphics[scale=0.9]{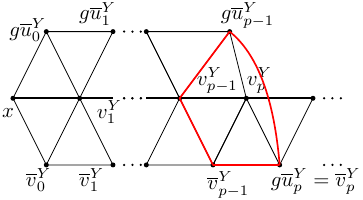}
\end{center}
\caption{Link of $v^Y_p$ (red).}\label{fig:ll1sys}
\end{figure}
Since $Y$ is systolic $\overline{v}^Y_{p-1}$ is incident to $g\overline{u}^Y_{p-1}$, as otherwise $\overline{v}^Y_p$ would be incident to $v^Y_{p-1}$, which contradicts the geodesity of $\gamma_2$ .

Clearly, the tuple 
$$
(u_{p-1}, u_{p-2},\ldots, u_1, x, v_1,\ldots, v_{p-1}, \overline{u}_{p-1}, \overline{u}_{p-2},\ldots, \overline{u}_1,\overline{u}_0,\overline{v}_0, \overline{v}_1,\ldots, \overline{v}_{p-1})
$$
satisfies the assumptions of Lemma \ref{laddyc6}.
Therefore one of the following holds:
\begin{enumerate}[(i)]
\item for some $1\leq s \leq p-1$ $(\overline{v}_{s}^Y, v_{s-1}^Y)\in Y$;
\item for some $1\leq s \leq p-1$ $(\overline{u}_{s}^Y, u_{s-1}^Y)\in Y$; 
\item $d_{Y_x}(g\overline{u}_{0}^Y,\overline{u}_{0}^Y)\leq 2$.
\end{enumerate}

In case (i) both $\overline{v}_M$ and $v_{s-1}$ are vertices belonging to the geodesic $\gamma_2$. Therefore $(v_{s-1},\ldots, v_N, \overline{v}_N)$ is a geodesic, but $(v_{s-1},\overline{v}_{s},\ldots, \overline{v}_N)$ is a shorter path between the same pair of vertices, a contradiction.

In case (ii) both $\overline{u}_N$ and $u_{s-1}$ are vertices belonging to the geodesic $\gamma_1$. Therefore $(u_{s-1},\ldots, u_N, \overline{u}_N)$ is a geodesic, but $(u_{s-1},\overline{u}_{s},\ldots, \overline{u}_N)$ is a shorter path between the same pair of vertices, a contradiction.

In case (iii) $\overline{u}_{0}^Y$ belongs to the link $x$ in $Y$ and $d_{Y_x}(g\overline{u}_{0}^Y,\overline{u}_{0}^Y)<3$, a contradiction.

In case 3) $g\overline{u}_p^Y=\overline{v}_{p+1}^Y$. But $g\overline{u}_p^Y$ is incident to $v_{p-1}^Y$. Both $\overline{v}_M$ and $v_{p-1}$ are vertices belonging to the geodesic $\gamma_2$. Therefore $(v_{p-1},\ldots, v_N, \overline{v}_N)$ is a geodesic, but $(v_{p-1},\overline{v}_{p+1},\ldots, \overline{v}_N)$ is a shorter path between the same pair of vertices, a contradiction.

In case 4) we have a double ladder with a cap in $Y$:

$$\{g\overline{u}^Y_{0},g\overline{u}^Y_{1},\ldots, g\overline{u}^Y_{N}=v_{N+1}^Y|x^Y, v^Y_{1},\ldots, v^Y_{N}|\overline{v}^Y_{0},\overline{v}^Y_{1}\ldots ,\overline{v}^Y_N\}$$ (see Figure \ref{fig:ll5sys}). 
\begin{figure}[H]
\begin{center}
\includegraphics[scale=0.9]{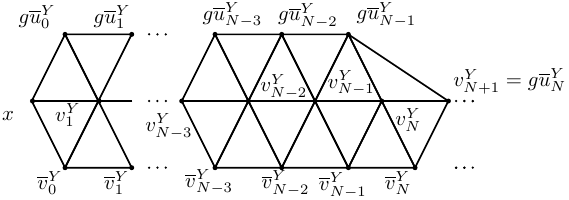}
\end{center}
\caption{Double ladder with a cap in $Y$. The triangle $[v_N^Y,v_{N+1}^Y,\overline{v}_{N}^Y]$ is a cap.}\label{fig:ll5sys}
\end{figure}

Clearly, the set 
$$
\{u_N, u_{N-1},\ldots, u_1, x, v_1,\ldots, v_{N}, \overline{u}_N, \overline{u}_{N-1},\ldots, \overline{u}_1,\overline{u}_0,\overline{v}_0,\overline{v}_1,\ldots, \overline{v}_{N}\}
$$
satisfies the assumptions of Lemma \ref{laddyc6}.
Again, each case following from this lemma results in a contradiction.

Therefore between any two vertices $x_n<u<v<y_m$ such that $\kappa_D(u)=\kappa_D(v)=1$ there is a vertex $w$, such that $u^Y<w^Y<v^Y$ and curvature of $w$ is at most $-1$. 
Therefore the sum of the curvature along any subpath of $\alpha$ is bounded by $1$ and we can use Lemma \ref{geo2} to show that $\alpha$ is a geodesic.
\end{proof}

\section{Rotations}\label{rotat}
Let $X$ be a simply connected \cftfs or \css small cancellation complex such that every $1$-cell of $X$ is contained in the boundary of a $2$-cell. Let $G$ be a finitely generated group acting on $X$ by automorphisms and assume that this action induces a free action on the $1$-skeleton $X^{1}$ of $X$. 

By the following proposition each group element acts by a rotation on each $2$-cell fixed by it.

\begin{pr}\label{rot}
Let $g\in G$ and $\widehat{x}\in \mathrm{Fix}_X(g)$. Then $g$ acts on $\widehat{x}$ by a rotation of finite order.
\end{pr}

\begin{proof}
Group $G$ acts on $X$ by automorphisms, therefore $g$ acts on $\widehat{x}$ by an isometry. Only possible isometries of $2$-cells are reflections and rotations. reflections does not act freely on $1$-skeleton, therefore $g$ acts by a rotation. 
Since $\widehat{x}$ is an $n$-gon for some $n$, then by the freeness of the action on $1$-skeleton $g$ has an order $m$ such that $m$ is a divisor of $n$. 
\end{proof}

We now consider \cftfs and \css cases separately.

\subsection{\cftfs case.}
Let $Y$ be the quadrization of $X$.

\begin{lm}\label{lmklcftf}
Let $\widehat{x}\in X$ be a $2$-cell fixed by $f\in G\setminus\{id\}$ and $x$ be a corresponding vertex in $Y$. 
Then there exists $k=k(f)$ such that for any $0$-cell $\widehat{y}$ belonging to $\widehat{x}$ we have $d_{Y_x}(f^ky,y)>2$.
\end{lm}

\begin{proof}
First observe that $d_{Y_x}(f^ky,y)\leq 2$ means that for some piece $p$ the intersection $p\cap f^kp$ is non-empty.

By Proposition \ref{rot}, $f$ is a rotation of some finite order $m$. Therefore, there exists $k_0$ such that $f^{k_0}$ is a `clockwise rotation through $\frac{2\pi}{m}$'. We claim that $k=\frac{mk_0}{2}$ for even $m$ and $k=\frac{(m-1)k_0}{2}$ for odd $m$ is as required.

Observe that if $m$ is even then $f^2k$ is a rotation through $2\pi$ and if it is odd then $f^3k$ is a rotation through at least $2\pi$. Indeed, since $m\geq 3$, we have $ 3\frac{(m-1)\pi}{m}\geq 2\pi$.
If for some piece $p$ the intersection $p\cap f^kp$ is non-empty, then we have three pieces $p, f^{k}p, f^{2k}p$ covering the whole boundary of $\widehat{x}$, see Figure \ref{fig:part2}. Contradiction with the $C(4)$ condition.
\end{proof}

\begin{figure}[h!]
\begin{center}
\includegraphics[scale=1.1]{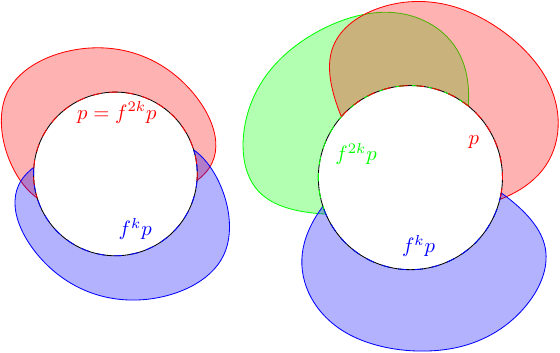}
\end{center}
\caption{Pieces covering whole boundary of the cell $\widehat{x}$ in the cases of even and odd $m$.}\label{fig:part2}
\end{figure}

\subsection{\css case.}
Let $\widehat{x}$ be a $2$-cell in $X$ and $v_1,v_2$ be a vertices from $\partial \widehat{x}$. By $(v_1,v_2)^{	\circlearrowright}$ we denote the clockwise path between $v_1$ and $v_2$ in $\partial \widehat{x}$. 
For any piece $p$ we assume that $p=(p_l,p_r)=(p_l,p_r)^{\circlearrowright}$.

Let $Y$ be a Wise complex of $X$. 
\begin{lm}\label{lmklcs}
Let $\widehat{x}\in X$ be a $2$-cell fixed by $f\in G$ and $x$ be a corresponding vertex in $Y$. 
Then either there exists $k=k(f)$ such that for any other $2$-cell $\widehat{y}$ we have $d_{Y_x}(f^ky,y)\geq 3$ or $f^3=id$.
\end{lm}

\begin{proof}
First observe that $d_{Y_x}(f^ky,y)< 3$ means that for some piece $p=\widehat{x}\cap\widehat{y}$ there's a piece $p'$ covering one of two arcs $(p_r,f^kp_l)^{\circlearrowright}$, $(f^kp_r,p_l)^{\circlearrowright}$.

By Proposition \ref{rot} $f$ is a rotation of some finite order $m$. Therefore, there exists $k_0$ (coprime with $m$) such that $f^{k_0}$ is a `clockwise rotation through $\frac{2\pi}{m}$'. We claim that $k=\frac{mk_0}{2}$ for even $m$ and $k=\frac{(m-1)k_0}{2}$ for odd $m$ is as required.

If for some piece $p$ there's a piece $p'$ covering one of two arcs $(p_r,f^kp_l)^{\circlearrowright}$, $(f^kp_r,p_l)^{\circlearrowright}$, then
if either $m$ is even, or $p'$ covers $(f^kp_r,p_l)^{\circlearrowright}$ we have that $f^{-k}p', p, p'$ and $f^kp$ covers whole boundary of the cell, a contradiction to the \css condition (see left part of the Figure \ref{fig:part2cs}).

If $m$ is odd and $p'$ covers $(p_r,f^kp_l)^{\circlearrowright}$.
we have three possible cases:
\begin{enumerate}
\item $p\cap f^{2k}p\neq \emptyset$;
\item $f^qp\cap f^rp=\emptyset$ for any $q\not\equiv r\pmod{m}$ and $\frac{3(m-1)}{2}>m$;
\item $\frac{3(m-1)}{2}\leq m$.
\end{enumerate} 
In the first case  $p, p', f^kp, f^kp'$ and $f^{2k}p$ cover the whole boundary, a contradiction (see central part of the Figure \ref{fig:part2cs}).
In the second case we have $\frac{2(m-1)}{2}<m<\frac{3(m-1)}{2}$ and in particular, since $k_0$ is coprime with $m$, we have $2k\not\equiv 0\pmod{m}$ and $3k\not\equiv 0\pmod{m}$. It follows that $f^{2k}p\cap p=f^{3k}p\cap p=\emptyset$, hence $f^{2k}p'\supset p$.
Thus $p', f^{k}p, f^kp',f^{2k}p$ and $f^{2k}p'$ cover the whole boundary, a contradiction (see right part of the Figure \ref{fig:part2cs}).
In the third case $m=3$.
\begin{figure}[h!]
\begin{center}
\includegraphics[scale=0.85]{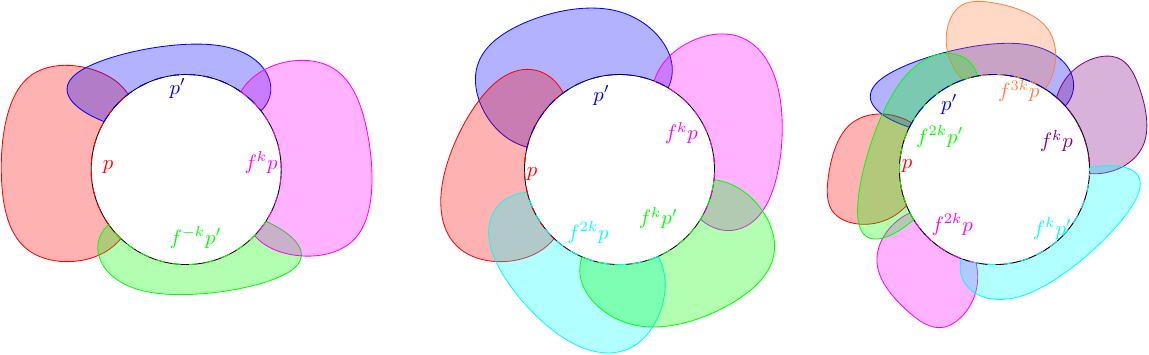}
\end{center}
\caption{Pieces covering whole boundary of the cell $\widehat{x}$.}\label{fig:part2cs}
\end{figure}
\end{proof}

\section{Lack of global fixed point implies existence of group element of infinite order}\label{inford}

In this section we study closer the fixed points of the action of $G$ on $X$.
First, we show that under assumptions from the previous section, each element of $G$ can fix at most one $2$-cell in $X$, equivalently at most one vertex in $Y$.

\begin{lm}\label{fixv}
	Let $g \neq 1$ be an element of $G$ such that $\mathrm{Fix}_Y(g)\neq\emptyset$. If $v \in \mathrm{Fix}_Y(g)$ then $\mathrm{Fix}_Y(g)=\{v\}$.
\end{lm}
\begin{proof}
	Assume that there is another $2$-cell such that $\widehat{v}'\in \mathrm{Fix}_X(g)$. 
	
	First, we consider the \cftfs case. 
	Let $\gamma := (v_0 :=v, v_1,\ldots, v_n:=v')$ be a geodesic in $Y$. 
	Let $k$ be given by Lemma \ref{lmklcftf} and $\alpha = \gamma \cup g^k\gamma$. By the choice of $k$ for any $u$ incident to $v$ we have $d_{Y_v}(g^ku,u)>2$.
	Therefore for every square $P\in Y$ containing $v$ in its boundary we have $P\cap g^kP=\{v\}$. By Lemma \ref{lad} $\alpha$ is a geodesic. Therefore $\gamma=g^k\gamma$. If $v\neq v'$ then $g^k$ fixes a vertex in $X$, contradiction with the freeness of the action on $1$-skeleton. 
	
	Now we consider \css case.
	We note here that because of case $g^3\neq id$ this case cannot be proved analogously to \cftfs case.
	Consider the set of all geodesics between $v$ and $v'$ in the complex $Y$. Since $Y$ is systolic, the set of all vertices incident to $v'$ belonging to these geodesics spans an $n$-simplex for some $n>0$. This simplex is fixed by $g$. Since $Y$ is a Wise complex, $g$ fixes an intersection of all $2$-cells corresponding to vertices of this simplex. That intersection is either a piece or a vertex. Contradiction with the freeness of the action on $1$-skeleton. 
\end{proof}

The aim of the rest of this section is to prove that for any two elements of $G$ that do not fix the same $2$-cell in $X$ there exists an element of $G$ of infinite order. By the following lemma such an element cannot fix any cell of $X$.

\begin{lm}\label{tr}
Let $g\in G$. If $\mathrm{Fix}_X(g)\neq \emptyset$ then $g$ has finite order.
\end{lm}

\begin{proof}
Let $\widehat{v}\in \mathrm{Fix}_X(g)$. It is a $2$-cell in $X$, therefore for some $n$ it is an $n$-gon. Let $\widehat{v}'$ be a vertex belonging to the boundary of $\widehat{v}$. Since $G$ acts by automorphisms on $X$ and $g$ fixes $\widehat{v}$, there exists $m>0$ such that $g^m\widehat{v}'= \widehat{v}'$. By the assumption of the freeness of the action on the $1$-skeleton $g^m$ is trivial.
\end{proof}

\begin{lm}\label{inf}
Let $X$ be either \cftfs or \css complex. 
If $\mathrm{Fix}_Y(f) \neq \mathrm{Fix}_Y(g)$ and:
\begin{itemize}
\item $X$ is a \cftfs complex and $k=k(f),l=k(g)$ are given by Lemma \ref{lmklcftf};
\item $X$ is a \css complex and $k=k(f),l=k(g)$ are given by Lemma \ref{lmklcs},
\end{itemize} 
 
then $f^kg^l$ has infinite order.
\end{lm}

\begin{proof}
Let $\gamma := (x_0 :=y, x_1,\ldots, x_n:=x)$ be a geodesic in $Y$. Consider the images of $\gamma$ under $(f^kg^l)^i$ and $(f^kg^l)^if^k$.
Since $G$ acts by automorphisms and $\gamma$ is a geodesic, $(f^kg^l)^i\gamma$ and $(f^kg^l)^if^k\gamma$ are also geodesics for any $i$.

We will prove that $$\alpha_i:=\bigcup\limits_{0\leq j\leq i} (f^kg^l)^j(\gamma\cup f^k\gamma)$$ is a geodesic for any $i$. In such a case $f^kg^l$ has an infinite order.

We argue by induction: first, observe that if $k,l$ are given by Lemma \ref{lmklcftf} (resp. by Lemma \ref{lmklcs}) then the conditions of Lemma \ref{lad} (resp. of Lemma \ref{ladc6}) are satisfied. It follows that $\alpha_0$ is a geodesic.

Now, we will show that the geodesity of $\alpha_{i+1}$ follows from the geodesity of $\alpha_i$.

To do that we first need to show that $\alpha_i\cup (f^kg^l)^{i+1}\gamma$ is a geodesic, this follows from Lemma \ref{lad} in the \cftfs case (resp. Lemma \ref{ladc6} in the \css case). Indeed, $\alpha_i$ and $(f^kg^l)^{i+1}\gamma$ satisfy the conditions of Lemma then their concatenation is a geodesic.

But then $\alpha_i\cup (f^kg^l)^{i+1}\gamma$ and 
$(f^kg^l)^{i+1}f^k\gamma$ also satisfy the conditions of Lemma \ref{lad} (resp. Lemma \ref{ladc6}). Therefore 
$$\alpha_{i+1}=\alpha_i\cup (f^kg^l)^{i+1}\gamma\cup (f^kg^l)^{i+1}f^k\gamma$$
is a geodesic. 

By induction, $\alpha_i$ is a geodesic for any $i$.
\end{proof}

This Lemma finishes the case of \cftfs complexes, and it remains to complete the case of \css complexes. We remind here that in the case of \css complexes, it is possible that for a given $f$, we cannot find $k$ such that $d_{Y_x}(f^ky,y)\geq 3$. 
In such a case by Lemma \ref{lmklcs} we know that $f^3= id$.

\begin{lm}\label{inffinc6}
Let $X$ be \css complex. If $G$ does not have a global fixed point, then there exists an element of infinite order in $G$. 
\end{lm}

\begin{proof}
Assume that $G$ does not have a global fixed point.
Let $f,g$ be a group elements such that $\{x\}=\mathrm{Fix}_Y(f) \neq \mathrm{Fix}_Y(g)=\{y\}$. We have three possible cases:
\begin{enumerate}
\item $f^3\neq id \neq g^3$;
\item exactly one of $f,g$ has order $3$;
\item $f^3=g^3=id$.
\end{enumerate}

In the first case, by Lemma \ref{lmklcs} there exist $k=k(f),l=k(g)$ such that the conditions of Lemma \ref{inf} are satisfied, therefore $f^kg^l$ has an infinite order.

In the second case without loss of generality we can assume that $f$ has the order $3$ and $g$ does not. Consider the conjugation of $g$ by $f$. It has same the order as $g$ and fixes the vertex $fy$. If $fy=y$ then $f$ fixes $y$, a contradiction. 
By Lemma \ref{lmklcs} there exists $l=k(g)$ such that the conditions of Lemma \ref{inf} are satisfied, therefore $fg^lf^{-1}g^l$ has an infinite order.

In the third case, observe that there is an element $h$ of the order not equal to $3$ in the subgroup $\langle f,g\rangle$. 
Indeed, if all elements of $\langle f,g\rangle$ have order $3$, then this subgroup is a quotient of the Free Burnside group $B(2,3)$ which is finite. This subgroup acts on a Wise complex, which is systolic. By \cite[Theorem C]{CO15}  a finite group acting on a systolic complex has a global fixed point, a contradiction to the fact that $\mathrm{Fix}_Y(f)\neq\mathrm{Fix}_Y(g)$.

Since $\mathrm{Fix}_Y(f)\neq\mathrm{Fix}_Y(g)$ then at least one of $\mathrm{Fix}_Y(f)$, $\mathrm{Fix}_Y(g)$ is not equal to $\mathrm{Fix}_Y(h)$.
Thus clearly at least one of $\mathrm{Fix}_Y(fhf^{-1})$, $\mathrm{Fix}_Y(ghg^{-1})$ is not equal to $\mathrm{Fix}_Y(h)$.
By Lemma \ref{lmklcs} there exists $k$ such that the conditions of Lemma \ref{inf} are satisfied either for $h$ and $fhf^{-1}$ or for $h$ and $ghg^{-1}$. It follows that at least one of $fh^kf^{-1}h^k, gh^kg^{-1}h^k$ has infinite order.
\end{proof}

\section{\cttss is CAT(0)}\label{catct}
In this section we prove Theorem \ref{thm:tF} and Corollaries \ref{thm:tE}-\ref{thm:tD}.

\begin{df}
A combinatorial $2$-complex $X$ is called a \textit{polygonal complex} if an intersection of any two closed cells of $X$ is either empty or exactly one closed cell. 
\end{df}

\begin{pr}
Any simply connected $T(6)$ complex is a polygonal complex.
\end{pr}
\begin{proof}
By the result of Pride \cite{Pride}, all pieces in $T(6)$ complex are of the length $1$ therefore each piece is exactly a closed $1$-cell. It follows that any non-empty intersection of two closed cells of $X$ consist of exactly one closed cell.
\end{proof}

We can view the edges of $X$ as segments of the length $1$ and the closed $2$-cells of $X$ as a regular Euclidean polygons of side length $1$. This induces a metric in $X$. It gives us a criterion for a polygonal complex to be CAT$(0)$. In the following definition the length of an edge in the link of $v$ is the angle in the corresponding polygon of $X$. 

\begin{df}
A polygonal complex $X$ with a metric $d$ satisfies \textit{the link condition} if for each vertex $v\in X$ every injective cycle in the link of $v$ has length at least $2\pi$.
\end{df}

Observe that \cttss complex $X$ does not necessarily have bounded size of $2$-cells, therefore we consider a complex $\mathfrak{X}$, which consists of barycentric subdivision of each cell in $X$.  
It is easy to see that $\mathfrak{X}$ is a triangle complex, and each triangle has one vertex corresponding to a center of a $2$-cell from $X$, one to a center of a $1$-cell and one which is a $0$-cell in $X$.

Clearly $\mathfrak{X}$ is a triangle complex, but it does not satisfy the link condition with metric induced by taking each $2$-cell to be a regular Euclidean triangle of the side length $1$.  
Therefore, we induce another metric $\mathfrak{d}$ in the following way: we take all $2$-cells to be Euclidean triangles with angle $\frac{\pi}{2}$ adjacent to a center of an $1$-cell of $X$, angle $\frac{\pi}{3}$ adjacent to a center of a $2$-cell of $X$ and angle $\frac{\pi}{6}$ adjacent to a $0$-cell from $X$.

\begin{proof}[Proof of Theorem \ref{thm:tF}]
It is enough to show that the complex $\mathfrak{X}$ with metric $\mathfrak{d}$ is CAT(0). Since $\mathfrak{X}$ with that metric has only one shape of $2$-cells, therefore by \cite[Lemma 5.6]{bri09} $\mathfrak{X}$ is CAT(0) as long as it satisfies the link condition i.e. we have to show that every injective cycle in each link has the length at least $2\pi$.

Let $v$ be a vertex from $\mathfrak{X}$. 
If $v$ is a $0$-cell in $X$ then each corner has length at least $\frac{\pi}{6}$. 
Each cycle in the link of $v$ consists of at least $12$ corners. Indeed, $X$ satisfies the condition $T(6)$, and each $2$-cell adjacent to $v$ in $X$ is replaced in the link by two triangles in $\mathfrak{X}$. Each corner has length at least $\frac{\pi}{6}$, thus each cycle has length at least $2\pi$.

If $v$ is a center of an $1$-cell, then each cycle in the link of $v$ has at least four corners, each of length $\frac{\pi}{2}$, so clearly each cycle has length at least $2\pi$.

If $v$ is a center of a $2$-cell, then each cycle in the link of $v$ has at least six corners, each of length $\frac{\pi}{3}$, so clearly each cycle has length at least $2\pi$.
\end{proof}

\begin{proof}[Proof of Corollary \ref{thm:tE}]
Let a finitely generated group $G$ act locally elliptically on a simply connected \cttss small cancellation complex $X$. Define $\mathfrak{X}$ and $\mathfrak{d}$ as above. Then $\mathfrak{X}$ has \emph{rational angles} with respect to $G$ in the sense of \cite[Definition 2.3]{NOP-D}. This follows from the fact that all triangles of $\mathfrak{X}$ have angles $\frac{\pi}{2},\frac{\pi}{3},\frac{\pi}{6}$ and from an observation on \cite[page 9]{NOP-D} just after \cite[Definition 2.3]{NOP-D}. Corollary follows from \cite[Theorem 1.1(iii)]{NOP-D}.
\end{proof}

\begin{proof}[Proof of Corollary \ref{thm:tD}]
Let $X$ be a simply connected \cttss small cancellation complex acted upon by $G$. Then, by Theorem \ref{thm:tF}, $G$ acts almost freely on the CAT$(0)$ complex $\mathfrak{X}$.
By \cite[Theorem A]{OsPrz22}, the group $G$ is virtually cyclic, or virtually $\mathbb{Z}^2$, or contains a nonabelian free subgroup.
\end{proof}

\begin{rem}
	Another way of proving Corollary \ref{thm:tD} is applying \cite[Main Theorem]{OsPrz21}. One observes that $\mathfrak{X}$ is \emph{reccurent} with respect to $G$ in the sense of \cite[Definition 2.1]{OsPrz21}.
	This is by Theorem \ref{thm:tF} and \cite[Remark 2.3]{OsPrz21}, because $\mathfrak{X}$ satisfies \cite[Definition 2.1(v)]{OsPrz21}. Moreover, $\mathfrak{X}$ admits a simplicial map to one triangle with angles $\frac{\pi}{2},\frac{\pi}{3},\frac{\pi}{6}$, whose restriction to each triangle of $\mathfrak{X}$ is an isometry.
	It follows that by \cite[Example 2.5]{OsPrz21} the complex $\mathfrak{X}$ satisfies \cite[Definition 2.1(i)-(iv)]{OsPrz21}. 
\end{rem}

Thanks to $\mathfrak{X}$ being CAT(0) we can also deduce the following Lemma that will be used in the proof of Theorem \ref{thm:tB}.

\begin{lm}\label{cttsfix}
Let $G$ act on \cttss complex $X$ by automorphisms such that the action induces a free action on the $1$-skeleton $X^1$ of $X$.
Let $g \neq 1$ be an element of $G$ such that $\mathrm{Fix}_X(g)\neq\emptyset$. If $v \in \mathrm{Fix}_X(g)$ then $\mathrm{Fix}_Y(g)=\{v\}$.
\end{lm}

\begin{proof}
Assume that $g$ has two fixed points $v\neq v'$. It is clear that $v$ and $v'$ are both centers of $2$-cells. The (unique) geodesic $\gamma$ between $v$ and $v'$ in $\mathfrak{X}$ is fixed by $g$. Since $\gamma$ has non-empty intersection with $X^1$, we get a contradiction.
\end{proof}

\section{Proofs of Theorems \ref{thm:tB} and \ref{thm:tC}}\label{sec: GASCC}

In our case of $G$ acting on $X$ by automorphisms in such a way that the action induces a free action on the $1$-skeleton $X^{1}$ of $X$, if the action is additionally locally elliptic, each element fixes a $2$-cell, equivalently, the center of a $2$-cell.

\begin{proof}[Proof of Theorem \ref{thm:tB}]
Assume that the action of $G$ on $X$ does not have a global fixed point.

In the \cftfs and \css cases we first observe the following.
Each $1$-cell of $X$ that is not contained in the boundary of a $2$-cell can be thickened to a $2$-cell to obtain a new $2$-complex $X'$ which deformation retracts to $X$.
The complex $X'$ is a \cftfs (or \cs) small cancellation complex. Moreover, the complex $X$ embeds into $X'$ and the action of $G$ is preserved, therefore the action of $G$ induces a free action on the $1$-skeleton of $X'$ and does not have a global fixed point. This allows us to use Lemmas from Sections \ref{negcuv}-\ref{inford}.
 
By Lemma \ref{fixv} an element of the group $G$ can fix at most one $2$-cell of $X'$. Assume that $f,g$ are elements of $G$ such that $\mathrm{Fix}(f)\neq \mathrm{Fix}(g)$. Then in \cftfs case by Lemma \ref{inf} there exist $k$, $l$ such that $f^kg^l$ has an infinite order. In \css case by Lemma \ref{inffinc6} there exist an element of an infinite order.
By Lemma \ref{tr} only elements of finite order 
can fix a $2$-cell, hence $G$ is not locally elliptic.

In the \cttss case by Corollary \ref{thm:tE} any finitely generated subgroup of $G$ has a global fixed point. From Lemma \ref{cttsfix} each nontrivial element of $G$ has exactly one fixed point.
As a consequence, any two nontrivial elements of $G$ fix the same point.

Therefore, in the \aotors cases there exists a $2$-cell fixed by all elements of $G$. It is an $n$-gon for some $n$. Then, by freeness of the group action on the $1$-skeleton of $X$, $G$ is finite and cyclic.
\end{proof}

We now pass to the proof of Theorem \ref{thm:tC}. For a group $G$ and a metric space $X$, a group action $\Phi:G\rightarrow Isom(X)$ is called \textit{proper} if for each $x\in X$ there exist a real number $r>0$ such that the set $\{g\in G | (B_r(\Phi(g)(x))\cap B_r(x))\neq 0\}$ is finite.
The group action $\Phi$ is called \textit{cocompact} if there exists a compact subset $K\subseteq X$ such that $\Phi(G)(K)=X$. We say that $\Phi$ is a \textit{geometric} action if it is both proper and cocompact.

A group $G$ is called \textit{artinian} if any descending chain of subgroups $G_1\supset G_2\supset \ldots$ becomes stationary, that is, $G_n=G_{n+1}=\ldots$ from some $n$ onwards.

Let $\mathcal{G}$ denote the class of all groups $G$ with the following three properties:
\begin{enumerate}[(i)]
\item $G$ does not include $\mathbb{Q}$ or the $p$-adic integers $\mathbb{Z}_p$ for any prime as a subgroup;
\item $G$ does not include the Pr\"ufer $p$ group $\mathbb{Z}(p^{\infty})$ for any prime as a subgroup.\end{enumerate}

\begin{pr}
If $G$:
\begin{enumerate}
\item is a CAT(0) group; or
\item is a Helly group; or
\item is a systolic group;
\end{enumerate}
then $G$ is in the class $\mathcal{G}$.
\end{pr}
\begin{proof}

Cases (1) and (2) are known by Prop 5.2 of \cite{keppeler2021automatic}.

Case (3). 
A systolic group is finitely generated, in particular it is countable, so it cannot contain an uncountable subgroup $\mathbb{Z}_p$.
It is known that an abelian subgroup of a systolic group is finitely generated, see \cite{OP18}, thus systolic group cannot have $\mathbb{Q}$ as a subgroup.
Furthermore, any systolic group contains only finitely many conjugacy classes of
finite subgroups: \cite[Corollary 1.3.]{przytycki_2008}, also \cite{CO15}. Therefore, there is a bound on the order of finite order elements in a systolic group, and therefore this group can not have Pr\"ufer $p$ group $\mathbb{Z}(p^{\infty})$ as a subgroup.
\end{proof}

\begin{proof}[Proof of Theorem \ref{thm:tC}]
By Theorem B of \cite{keppeler2021automatic} the statement holds for subgroups of any group belonging to the class $\mathcal{G}$ whose torsion subgroups are artinian.
In the \cftfs case, by Theorem 6.18 of \cite{Helly}, $G$ is Helly because it acts geometrically on $X$.
In the \css case, $G$ is systolic because it acts geometrically on the Wise Complex, which is systolic.
In the \cttss case, by Theorem \ref{thm:tF}, $G$ acts geometrically on a CAT(0) complex, therefore it is CAT(0).

In the \cftfs and \css cases, if $H$ is a torsion subgroup of $G$ then its action is locally elliptic on $X$, and by Theorem \ref{thm:tB} we know that the subgroup $H$ has to be finite.
In \cttss case, if $H$ is a torsion subgroup of $G$ then it acts almost freely on $X$, and by Theorem \ref{thm:tD} it is finite.
It is clear that finite groups are artinian.
\end{proof}

\section{Proof of Theorem \ref{thm:tA}}\label{sec: PTA}

Let $\langle X| R\rangle$ be a presentation of a group $G$. The \textit{presentation complex} of $\langle X| R\rangle$ is formed by taking a unique $0$-cell, adding a labeled oriented $1$-cell for each generator, and then attaching a $2$-cell along the closed combinatorial path corresponding to each relator.

The \textit{Cayley complex} of $G$ with respect to the presentation $\langle X|R\rangle$ (denoted $Cayley(G,X,R)$) is constructed in the following way. Let the set of vertices of $Cayley(G,X,R)$ consist of all the elements of $G$. Then, at each vertex $g\in G$, insert a directed edge from $g$ to $gx$ for each of the generators $x\in X$. The translation of any relator $r\in R$ by any element of the group $G$ gives a loop in the graph. We attach a $2$-cell to each such loop. The $1$-skeleton of the Cayley complex is a directed graph known as the \textit{Cayley graph}. It is known \cite[Section 1.3]{AH} that the Cayley complex is a universal cover of the presentation complex.

The presentation $\langle X| R\rangle$ of the group $G$ is a $C(p)$--$T(q)$ \textit{small cancellation presentation} if its presentation complex is a $C(p)$--$T(q)$ complex. In such a case, since the Cayley complex with respect to the presentation $\langle X| R\rangle$ is a universal cover of the presentation complex, it is a $C(p)$--$T(q)$ complex as well \cite{MW}. 

A \textit{spherical diagram} $S$ is a $2$-sphere $\mathbb{S}^2$ with a structure of a combinatorial $2$-complex.
As in the case of disc diagrams, a diagram $S$ \textit{in} $X$ is $S$ along with a combinatorial map from $S$ to $X$ denoted by $S\rightarrow X$. 

A presentation of a group is \textit{aspherical} if there are no 'non-trivial' spherical diagrams in $Cayley(G,X,R)$ in the sense of \cite[III.10, p.156]{ls}. The following theorem states one of the known properties of groups with a small cancellation presentation.

\begin{thm}\cite[Theorem 4]{huebschmann1979cohomology}
Any \aotors small cancellation presentation is aspherical.
\end{thm}

Before giving a proof of Theorem \ref{thm:tA} we need to state a theorem of Huebschmann concerning groups with aspherical presentations.

\begin{thm}\label{hueth3}\cite[Theorem 3]{huebschmann1979cohomology}
Let $G$ be a group with an aspherical presentation $\langle X|R\rangle$. If $x\in G$ is an element of order $1<s<\infty$, then there is a relator $r=z^{q_r}_r$ with $s|q_r$ such that $x$ is conjugate to $z^{q_r/s}_r$.
\end{thm}

\begin{proof}[Proof of Theorem \ref{thm:tA}]
Let $Cayley(G,X,R)$ be either a \cttss, a \cftfs or a \css complex.
Assume that $H$ is a torsion subgroup of $G$.
The action of the group $G$ on the $0$-skeleton of the Cayley complex is free. As the $1$-skeleton of $Cayley(G,X,R)$ is a directed graph, the action of $G$ on the $1$-skeleton is free. This property is inherited by any subgroup of $G$, in particular $H$. Since $H$ is a torsion group, by Theorem \ref{hueth3} each of its elements is conjugate to a root of some relator. Obviously, a root of a relator fixes the $2$-cell corresponding to this relator in the Cayley complex, therefore the action of $H$ on $Cayley(G,X,R)$ is locally elliptic. It follows from Theorem \ref{thm:tB} that $H$ is a finite cyclic group. 
\end{proof}

\bibliographystyle{alpha}
\bibliography{references}{}
\end{document}